\documentclass[11pt,oneside]{amsart}

\usepackage[margin=1in]{geometry}

\usepackage{amssymb, amsmath}
\usepackage{float}

\usepackage{booktabs}
\usepackage{graphicx}
\usepackage{color}
\usepackage{pinlabel}
\usepackage{mathtools,xparse}
\usepackage{tikz}
\usetikzlibrary{arrows,positioning,shapes,fit,calc}

\usepackage{enumerate}

\usepackage[colorlinks = true,
            linkcolor = red,
            urlcolor  = red,
            citecolor = red,
            anchorcolor = red]{hyperref}
\usepackage{amsrefs}

\usepackage[pagewise]{lineno}

\usepackage{pinlabel}

\theoremstyle{plain}
\newtheorem{theorem}{Theorem}[section]

\newtheorem{corollary}[theorem]{Corollary}
\newtheorem{lemma}[theorem]{Lemma}

\theoremstyle{definition}
\newtheorem{remark}[theorem]{Remark}
\newtheorem{definition}[theorem]{Definition}

\theoremstyle{definition}

 \usepackage[abs]{overpic}

\begin{document}
   \title[The stick number of rail arcs]{The stick number of rail arcs}
   \author{Nicholas Cazet}
   
   \begin{abstract}

Consider two parallel lines $\ell_1$ and $\ell_2$ in $\mathbb{R}^3$. A rail arc is an embedding of an arc in $\mathbb{R}^3$ such that one endpoint is on $\ell_1$, the other is on $\ell_2$, and its interior is disjoint from $\ell_1\cup\ell_2$. Rail arcs are considered up to rail isotopies, ambient isotopies of $\mathbb{R}^3$ with each self-homeomorphism mapping $\ell_1$ and $\ell_2$ onto themselves. When the manifolds and maps are taken in the piecewise linear category, these rail arcs are called stick rail arcs. 

The stick number of a rail arc class is the minimum number of sticks, line segments in a p.l. arc, needed to create a representative. This paper calculates the stick numbers of rail arcs classes with a crossing number at most 2 and uses a winding number invariant to calculate the stick numbers of infinitely many rail arc classes.

 Each rail arc class has two canonically associated knot classes, its under and over companions. This paper also introduces the rail stick number of knot classes, the minimum number of sticks needed to create a rail arcs whose under or over companion is the knot class. The rail stick number is calculated for 29 knot classes with crossing number at most 9. The stick number of multi-component rail arcs classes is considered as well as the lattice stick number of rail arcs. 

 \end{abstract}

\maketitle

\section{Introduction}

Manifolds and maps are smooth or piecewise linear depending on the context.   Knotoids were introduced by Turaev in \cite{turaev2012knotoids}.

\begin{definition} A {\it planar knotoid diagram} $\kappa\subset\mathbb{R}^2$ is a generic immersion of the unit interval with crossing information given at each double point. Furthermore, the image of 0 and 1 are distinct and called the {\it tail} and {\it head} of $\kappa$. A {\it trivial} planar knotoid diagram is an embedding of the unit interval. 

\end{definition}

The three Reidemeister moves are defined on planar knotoid diagrams away from the head and tail, the supporting disk of the move is disjoint from either endpoint.

\begin{definition}

Consider the equivalence relation on planar diagrams generated by Reidemeister moves and planar isotopy.  The equivalence classes of this relation are called {\it planar knotoids}.

\end{definition}

 \noindent In general, knotoids can be defined on any orientable surface $\Sigma$ with knotoids in $\Sigma=S^2$ and $\Sigma=\mathbb{R}^2$ referred to as {\it classical}.  
 
 Planar knotoids are in correspondence with rail arc classes.

 \begin{definition} For two parallel lines $\ell_1$ and $\ell_2$ in $\mathbb{R}^3$ called {\it rails}, a {\it rail arc} is an embedding of a simple arc $r$ with one endpoint on $\ell_1$, the other on $\ell_2$, and an interior disjoint from $\ell_1\cup\ell_2$. Two rails arcs $r_1$ and $r_2$ are {\it rail isotopic} if there is an ambient isotopy of $\mathbb{R}^3$ taking $r_1$ to $r_2$ while mapping the rails onto themselves throughout the isotopy. A {\it rail arc class} will mean a rail isotopy class of rail arcs. \end{definition}
 
 \noindent Generically projecting a rail arc onto a plane perpendicular to its rails gives a planar knotoid diagram, and  planar knotoid diagrams lift to rail arcs as knot diagrams lift to knots in 3-space, see Figure \ref{fig:fig8}.

 \begin{theorem}[G\"ug\"umc\"u and Kauffman '17 \cite{gugumcu2017new}] Two rail arcs are rail isotopic if and only if their projections to planar knotoid diagrams represent the same planar knotoid.

 \end{theorem}
 
 \noindent This theorem establishes a one-to-one correspondence between rail arc classes and planar knotoids.
 
 Planar knotoid diagrams are nearly knot diagrams.

\begin{definition} Connect the endpoints of a planar knotoid diagram $\kappa$ with another generic immersion of the unit interval $\alpha$ that meets $\kappa$ transversely. A knot diagram is created with the addition of crossing information at the intersection of $\alpha$ and $\kappa$. If $\alpha$ is taken to be under whenever it intersects $\kappa$, then the resulting knot class that the diagram represent is called the {\it under companion} of $\kappa$. Likewise, if $\alpha$ is taken to be over whenever it intersects $\kappa$, then the resulting knot class is called the {\it over companion} of $\kappa$. 

\end{definition}

\noindent The under and over companions of a planar knotoid diagram do not depend on the choice of $\alpha$, as long as all of its crossings with $\kappa$ are all under or all over; $\alpha$ is referred to as an {\it under pass} or {\it over pass} in these cases. Equivalent planar knotoid diagrams have equivalent under and over companions. Therefore, each planar knotoid has a well-defined under and over companion.

Another diagrammatic way of approaching rail arcs was introduced by Kodokostas and Lambropoulou in \cite{kodokostas2019rail}.

\begin{definition} Projecting a rail arc onto a plane containing the rails gives a {\it planar rail knotoid diagram}. The {\it under companion} of a planar rail knotoid diagram is the knot diagram represented by connecting the tail with arc than runs down the first rail, with crossing information inherited from the rail, until a perpendicular arc can be added to connect with the second rail and a third arc is added that connects the second arc to the head along the second, rail with crossing information inherited from the rail. An {\it over companion} is defined likewise with arcs running up the rails. \end{definition}

\noindent The under and over companions of a planar rail knotoid diagram  are the under and over companions of the planar knotoid diagram when the rail arc is projected onto a plane perpendicular to its rails.   See Figure \ref{fig:fig8} for the relationship between rail arcs, their planar rail knotoid diagrams, and their planar knotoid diagrams. There are several local moves, including the Reidemeister moves, that generate an equivalence relation on planar rail knotoid diagrams such that two diagrams are equivalent if and only if their associated rail arcs are rail isotopic, see Figure 8 of  \cite{kodokostas2019rail}.

\begin{figure}[h]

\includegraphics[scale=.58]{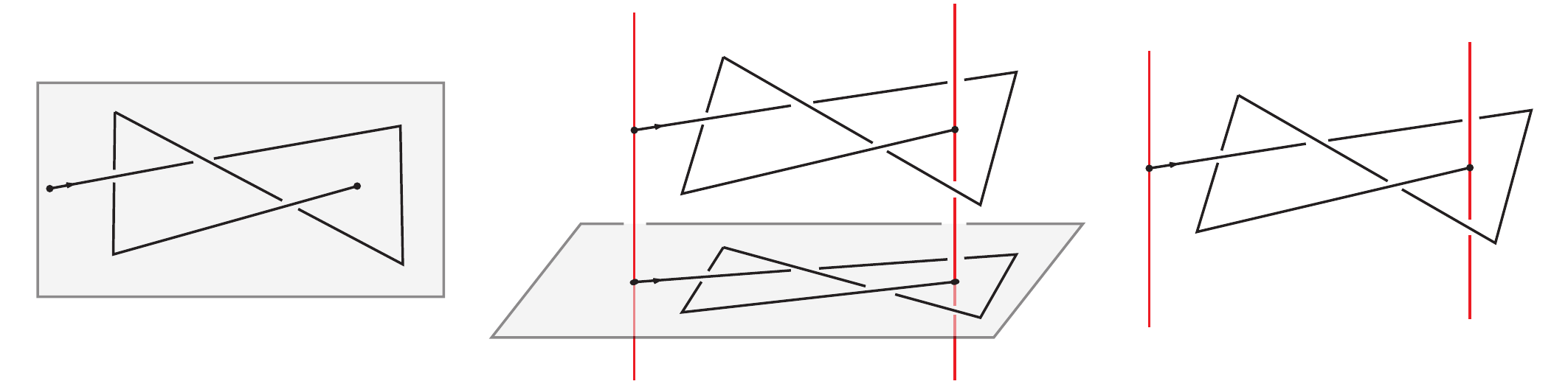}

\label{fig:fig8}

\caption{A stick rail arc with its planar knotoid projection on the left and rail knotoid projection on the right.}
\end{figure}

The stick number of stick knots motivates this paper's study of rail arcs.

\begin{definition} A {\it stick knot} is a p.l. knot in $\mathbb{R}^3$ seen as a finite union of line segments called {\it sticks}. Stick knots are equivalent if there is a p.l. isotopy of $\mathbb{R}^3$ taking one to the other. Equivalent stick knots differ by a sequence of triangle moves.
\end{definition}

 \begin{definition} The {\it stick number} $s[K]$ of a p.l. knot class $[K]$ is the minimum number of sticks needed to construct a representative.
 \end{definition}

\noindent As with most min/max knot invariants, the stick number is difficult to compute.  
  
Jin and Kim \cite{jin}, Negami \cite{Negami1987}, and Randell \cite{randell} showed that $s[3_1]=6$, $s[4_1]=7$. Randell additionally showed that every other non-trivial knot class must have a stick number at least 8. Negami \cite{Negami1987}, Huh and Oh \cite{upper}, and Calvo \cite{calvo} bounded the stick number of a knot class by its crossing number. Adams et al. studied the stick number of knots and links under various restrictions concerning the length of the sticks, the angles between sticks, and placements of the vertices \cite{Adams1997}. They also gave the stick numbers of infinitely many knots, the $(p,p-1)$-torus knots. Huh, Oh, and  No provided upper bounds on the stick number of 2-bridge knots \cite{2bridge}. A summary of the key results in this area, as well as a table giving the best known bounds on the stick numbers of knots with at most 10 crossings, is found in \cite{eddy}. For example,  see      \cites{spherical,blair,ramsey,spatial,equilateral}.

The stick number is also considered for stick knots confined to the simple cubic lattice.
  
  \begin{definition} Stick knots constrained to the {\it simple cubic lattice} $\mathbb{L}^3=\mathbb{R}\times\mathbb{Z}\times\mathbb{Z}\cup\mathbb{Z}\times\mathbb{R}\times\mathbb{Z}\cup\mathbb{Z}\times\mathbb{Z}\times\mathbb{R}$ are called {\it lattice knots}. Lattice knots are a subclass of p.l. knots such that two lattice knots are equivalent if there is a p.l. isotopy taking one to the other. They admit a stick number called the {\it lattice stick number}, denoted $s_{CL}[K]$, where the minimum is taken over all lattice representatives of the class $[K]$ . \end{definition}

 \noindent Diao \cite{diao} proved $s_{CL}[3_1]=12$, Promislow and Rensburged \cite{rens} proved that $s_{CL}[9_{47}]=18$, and Huh and Oh 
  proved that $s_{CL}[4_1]=14$ in \cite{latticehuh}. Huh and Oh proved  that $3_1$ and $4_1$ are the only knots with a lattice stick number less than 15 \cite{Huh_2010}. Adams et al. proved that $s_{CL}[8_{20}]=s_{CL}[8_{21}]=s_{CL}[9_{46}]=18$ and $s_{CL}[\ \text{$(p,p+1)$-torus knot}\ ]=6p$ \cite{adams}. They also gave results on minimal lattice stick conformation of knot sums, satellite knots, and links.

With the previously cited interest in the stick number of knots, it is natural to consider the stick number of rail arcs.

\begin{definition}
A {\it stick rail arc} is a rail arc taken in the p.l. category, i.e. a rail arc that is a finite union of line segments called {\it sticks}. Two stick rail arcs are equivalent if there is p.l. rail isotopy taking one to the other. This implies that equivalent stick rail arcs differ by a sequence of triangle moves. 

\end{definition}

\begin{definition}

The {\it stick number} of a rail arc class is the minimum number of sticks needed to created a p.l. representative. The stick number of a rail arc class $[r]$ is denoted $s[r].$

\end{definition}

\begin{definition}
A {\it lattice rail arc} is a stick rail arc whose arc and rails are constrained to the simple cubic lattice.  The minimum number of stick needed to create a lattice representatives of a rail arc class $[r]$ is its {\it lattice stick number} and is denoted $s_{CL}[r]$.

\end{definition}

Stick knots are used in the study of DNA and ringed polymers, see \cite{grid} \cite{deguchi}, \cite{polygons}, \cite{Janse_van_Rensburg_2012}. DNA topologists are interested in the minimal length of lattice representatives \cite{Scharein_2009}. Planar knotoids and rail arcs are being researched in the context of open-knotted protein chains, see  \cite{barbensi2021f},  \cite{dorier2018knoto}, \cite{goundaroulis2017studies}, \cite{polym}, \cite{gugumcu2017knotoids}. Altogether, this suggests the practicality of studying the stick number of rail arc classes. This paper is concerned with the theoretical interest of finding the minimum number of sticks needed to create a representative of a rail arc class. 

Theorems \ref{thm:lower} and \ref{thm:4} establish that every non-trivial rail arc class requires at least 4 sticks to construct a representative and there are only three rail arc classes with a stick number at most four.  Section \ref{stickindex} calculates the stick numbers of all rail arc classes with a crossing number at most two. In doing so, a winding invariant is defined that is also used to calculate the stick numbers of an infinite family. Section \ref{railindex} relates the stick number of rail arc classes and the stick number of knot classes by defining the {\it rail stick number}. This is the minimum number of sticks among all stick rail arcs whose under or over companion is a given knot. Theorem \ref{thm:knots} calculates the rail stick number of 29 knot classes with a crossing number at most 9.  To conclude, Section \ref{cubic} provides a few results on the lattice stick number of rail arcs and the rail lattice stick number of knots and links.

\section{Planar Knotoid / Rail Arc Classification}

The {\it crossing number} of a rail arc class is the minimal number of crossings among all representative planar knotoid diagrams. A planar knotoid diagram is {\it normal} if its tail is adjacent to the infinite region.  This section discusses the classification of prime rail arc classes with few crossings. Products of knotoids is defined in \cite{turaev2012knotoids}. Suppose that $\kappa_1$ and $\kappa_2$ are planar knotoid diagrams such that $\kappa_2$ is normal. Concatenating $\kappa_1$ at its head with a copy of $\kappa_2$ at its tail in a small neighborhood of the former's head produces the {\it product planar knotoid diagram $\kappa_1\kappa_2$,} see Figure \ref{fig:vertexproduct}. The {\it product planar knotoid} $[\kappa_1][\kappa_2]=[\kappa_1\kappa_2]$ is independent of representative used in the product thus is well-defined.

\begin{figure}[h]

\begin{overpic}[unit=.434mm,scale=.9]{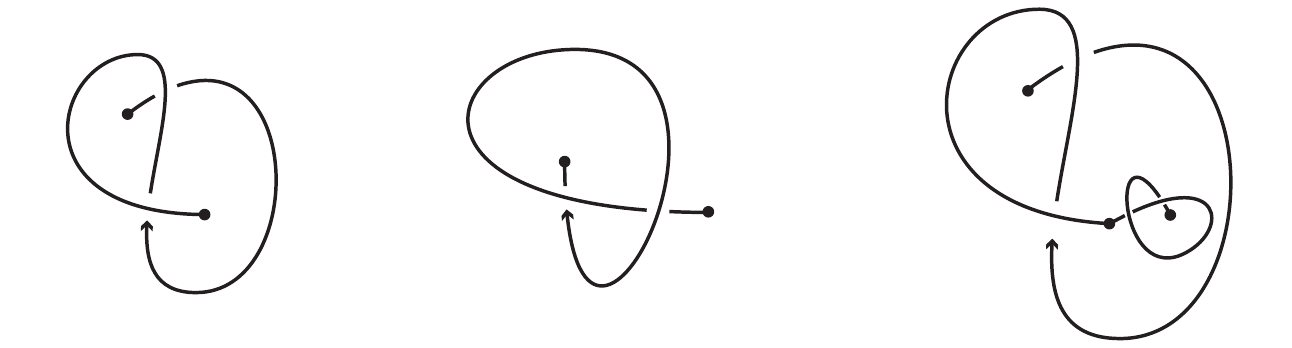}\put(34,0){$\kappa_1$}\put(117,0){$\kappa_2$}\put(202,0){$\kappa_1\kappa_2$}\put(78,34){$\cdot$}\put(173,34){=}

\end{overpic}

\label{fig:vertexproduct}

\caption{A product of $\kappa_1$ and $\kappa_2$.}
\end{figure}
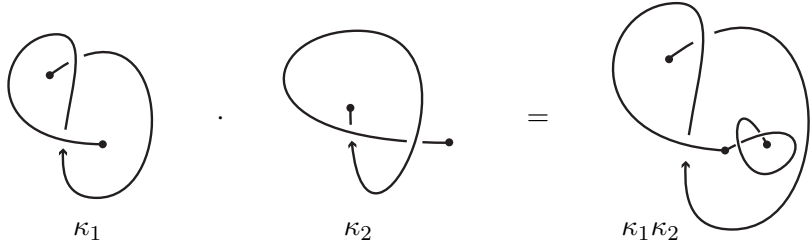

A knotoid is {\it prime} if it cannot be written as a product of two non-trivial knotoids. A rail arc is {\it prime} if its corresponding planar knotoid is prime, i.e. whenever the planar knotoid is written as a product one component is trivial.

Goundaroulis, Dorier, and Stasiak systematically classified prime planar knotoids, thus prime rail arcs, with up to five crossings \cite{goundaroulis2019systematic}. This classification is up to orientation and symmetry-related involutions mir($\cdot$), sym($\cdot$), and rot($\cdot$), see Figure \ref{fig:inv}. The {\it mirror reflection}, mir($\cdot$), changes the each crossing and is equivalent to reflecting the associated rail arc across a plane perpendicular to its rails. {\it Symmetry}, sym($\cdot$), reflects a knotoid diagram with respect to the vertical line $0\times \mathbb{R}\subset \mathbb{R}^2$.  The composition of mirror reflection and symmetry is called {\it rotation}, rot($\cdot$). 

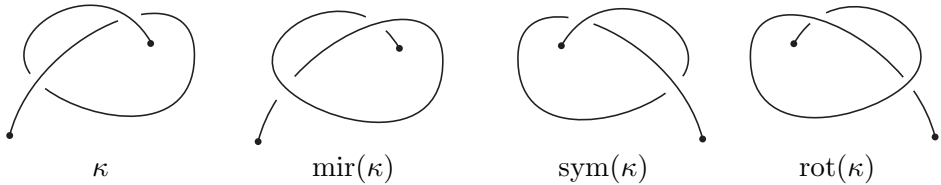
\begin{figure}[h]

\begin{overpic}[unit=.395mm,scale=.6]{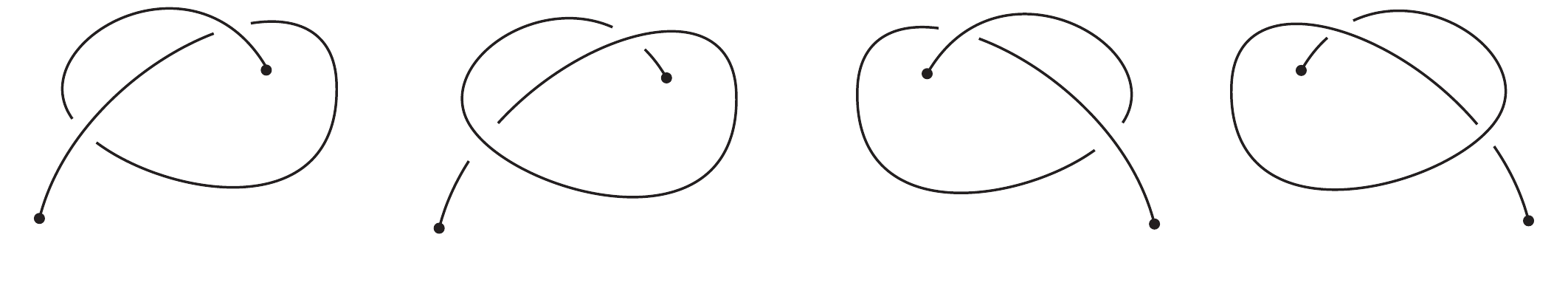}\put(36,5){$\kappa$}\put(110,5){$\text{mir}(\kappa)$}\put(193,5){$\text{sym}(\kappa)$}\put(274,5){$\text{rot}(\kappa)$}

\end{overpic}

\label{fig:inv}

\caption{Involutions of a planar knotoid diagram.}
\end{figure}

\noindent They introduced a two-number notation of knotoids, analogous to the Rolfsen notation of  knots and links. The first number represents the crossing number of the planar knotoid  and the second number represents an index to distinguish planar knotoids  with the same crossing number. The two-number notation  will be used for planar knotoids, rail arc classes, and knots. The context of its use will imply its representation. The first seven planar knotoids of their classification, all planar knotoids up to two crossings, are shown in Figure \ref{fig:2crossing}.

\begin{figure}[h]

\begin{overpic}[unit=.434mm,scale=.87]{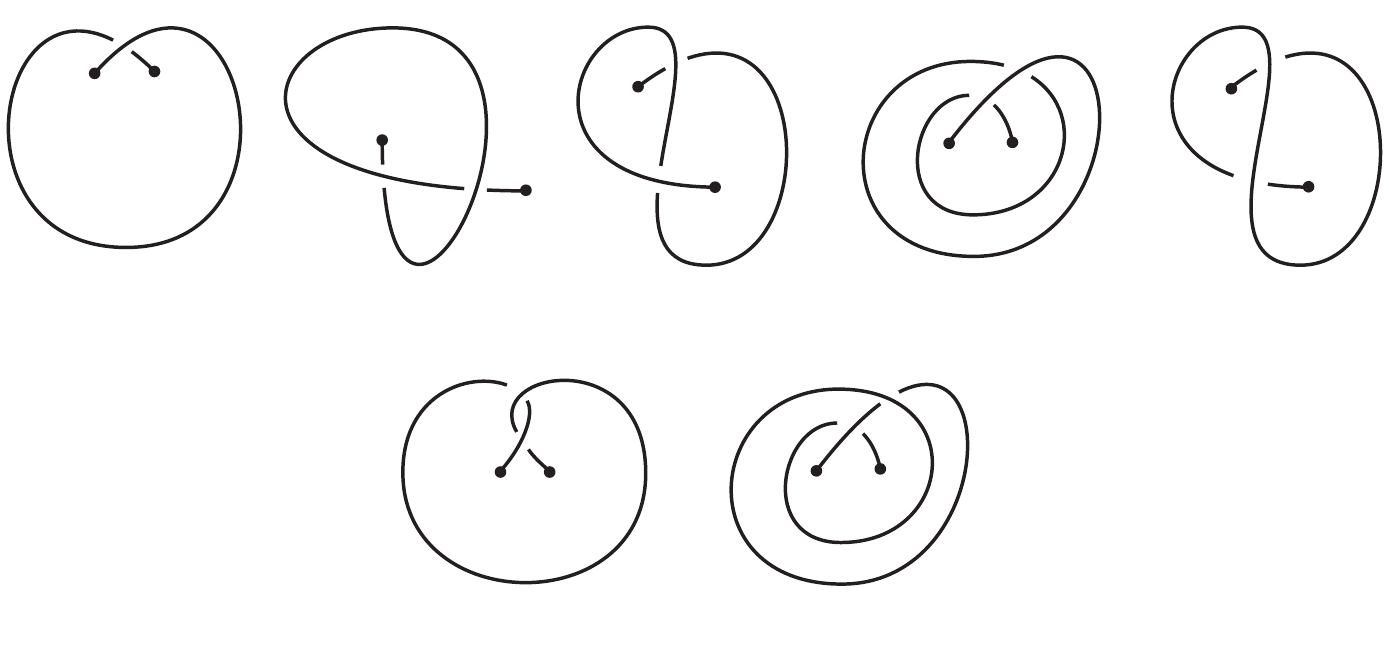}\put(24,67){$1_1$}\put(83,67){$2_1$}\put(140,67){$2_2$}\put(195,67){$2_3$}\put(260,67){$2_4$}\put(107,2){$2_5$}\put(170,2){$2_6$}

\end{overpic}

\caption{Planar knotoids, up to orientation and involution, with a crossing number at most 2.}\label{fig:2crossing}
\end{figure}
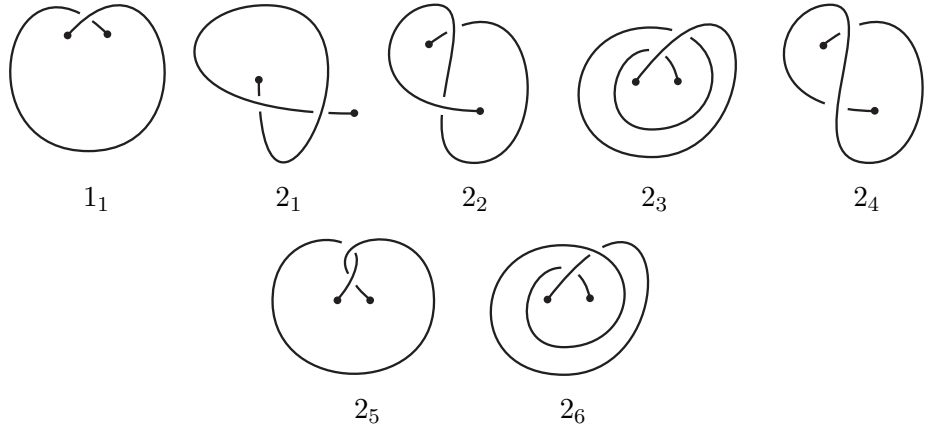

While the rail arc classes generated by reversing orientation and applying involutions may not be rail isotopic, classes related by orientation reversal or involution have the same stick number. Therefore, calculating the stick numbers of the classes in Figure \ref{fig:2crossing} determines the stick numbers of all rail arc classes up to 2 crossings.

\section{The stick number of Rail Arcs}
\label{stickindex}

This section calculates the stick number of all rail arc classes with at most two crossings. To complete the calculation a winding number invariant is defined for a certain subclass of rail arcs. This invariant is also used to calculate the stick numbers of an infinite family of rail arc classes, Corollary \ref{cor:winding}.

\begin{theorem}

For any non-trivial rail arc $r$, i.e. $r$ is not rail isotopic into a plane, \[4\leq s[r].\]

\label{thm:lower}
\end{theorem}

\begin{proof}

A one stick rail arc is trivial.  Consider the planar knotoid projection of $r$ in a plane perpendicular to its rails.  It is not possible to create a crossing in a planar knotoid diagram with only two sticks, and a knotoid diagram without crossings is trivial. Three sticks can create at most one crossing in a planar knotoid diagram. A one crossing knotoid diagram is possible only if the stick containing the head intersects the stick containing the tail. The addition of the third stick creates a triangular region in the plane that cannot contain the head nor tail. Because this triangular region is away from the projection of the rails, a Reidemeister I move will eliminate this sole crossing. Therefore, a three stick rail arc must be trivial. 

\end{proof}

\begin{corollary}

A rail arc class has a stick number of 1 if and only if it is trivial, and the rail arc classes $1_1$ and $2_1$ have a stick number of 4.

\end{corollary}

\begin{proof}

This is a consequence of Theorem \ref{thm:lower} and Figure \ref{fig:4stick}.

\end{proof}

\begin{figure}[h]

\includegraphics[scale=.6]{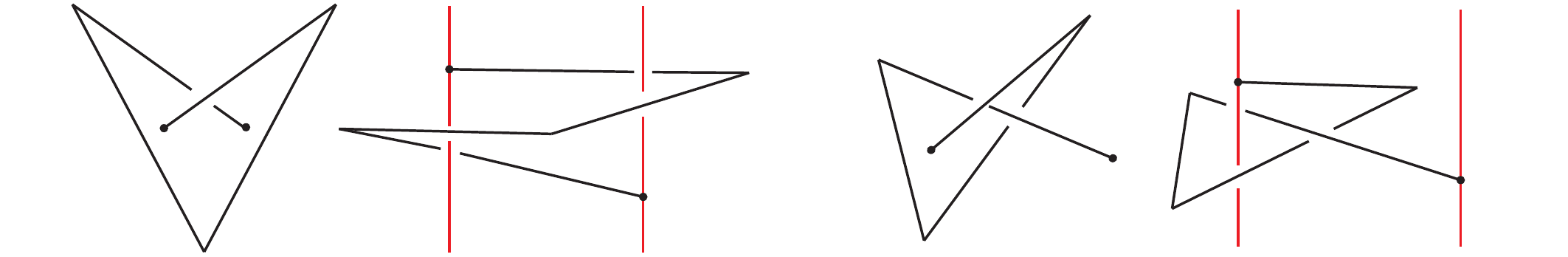}

\caption{Minimal stick representatives of $1_1$ and $2_1$.}
\label{fig:4stick}
\end{figure}

\begin{theorem}

The only rail arc classes with a stick number of 4 are $1_1$ and $2_1$.

\label{thm:4}
\end{theorem}

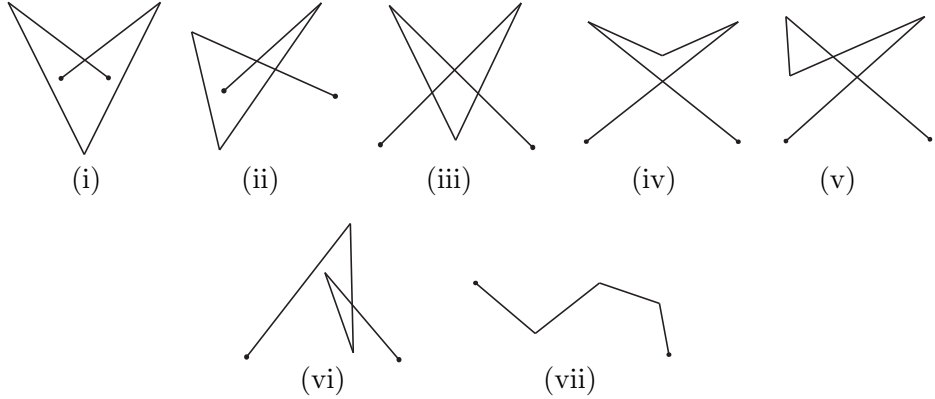
\begin{figure}[h]

\begin{overpic}[unit=.39mm,scale=.6]{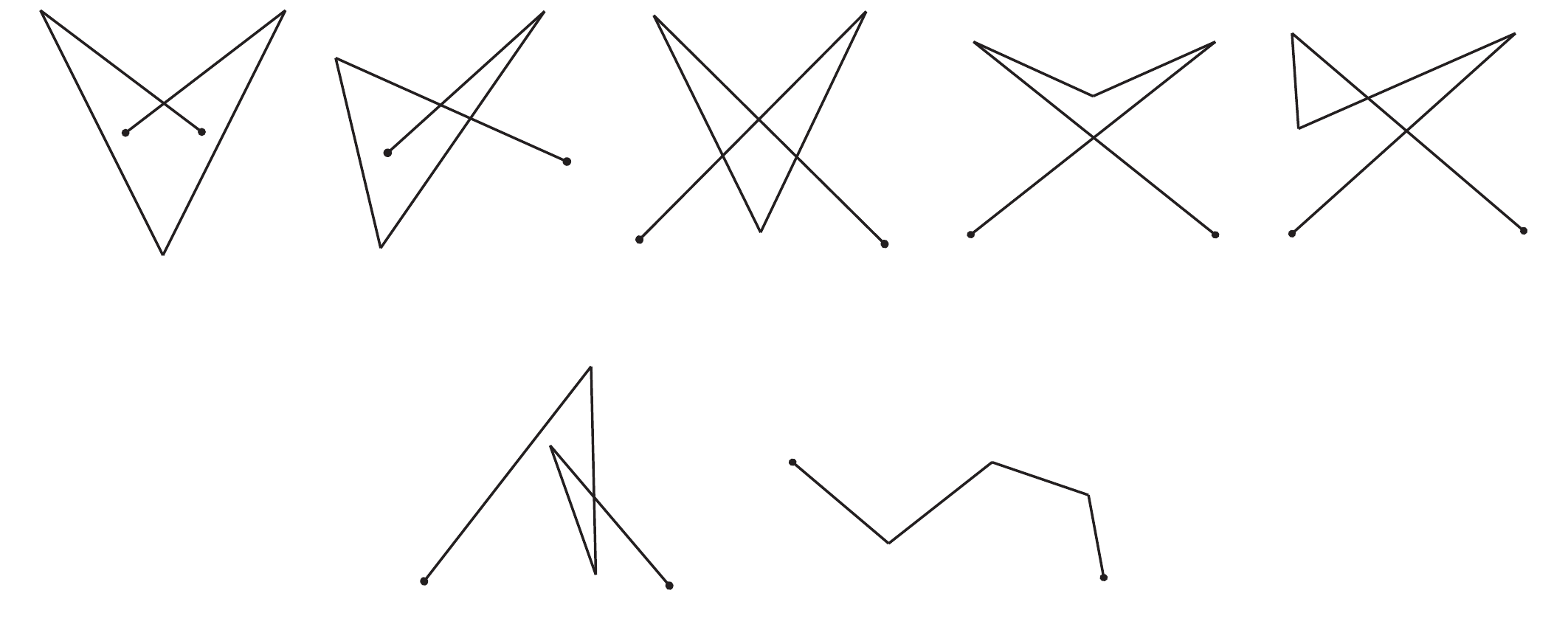}\put(30,70){(i)}\put(88,70){(ii)}\put(151,70){(iii)}\put(221,70){(iv)}\put(285,70){(v)}

\put(108,2){(vi)}\put(191,2){(vii)}
\end{overpic}

\caption{Planar knotoid projections of four stick rail arcs.}
\label{fig:fourstick}
\end{figure}

\begin{proof}

Consider the generic planar knotoid projections of a p.l. rail arc with four sticks. Order the sticks following the orientation from head to tail. Up to involution and planar isotopy, there are seven classes of projections to consider, see Figure \ref{fig:fourstick}. This enumeration is done by considering possible intersections among the projected sticks and whether a planar region is created containing the head and/or tail. For example, a one sole crossing diagram can occur if the first stick intersects the fourth stick, the first stick intersects the third stick, or the second stick intersects the fourth stick. The latter two cases are redundant by involutions and are represented by Figure \ref{fig:fourstick}(vi). In the case that the first stick intersects the fourth stick, either a planar region is created away from the rails or containing both rails, Figures \ref{fig:fourstick}(iv) and \ref{fig:fourstick}(i). Similar arguments enumerate the seven isotopy classes represented in Figure \ref{fig:fourstick}.

Regardless of crossing information, Figure \ref{fig:fourstick}(i) represents $1_1.$ If the crossings of Figure \ref{fig:fourstick}(ii) alternate, then it represents $2_1$. If the crossings of Figure \ref{fig:fourstick}(ii) do not alternate, then it represents the trivial rail arc. If the crossings of Figure \ref{fig:fourstick}(iii) alternate, then a planar isotopy will take the diagram to miss the line containing the head and tail. Connecting the head and tail by a line segment gives a 5 stick knot diagram of the trefoil. Since no non-trivial knot has a stick number less than 6, this contradiction implies that Figure \ref{fig:fourstick}(iii)'s crossing cannot alternate. If the crossings of Figure \ref{fig:fourstick}(iii) do not alternate, then the diagram is trivial after Reidemeister I and II moves. Figures \ref{fig:fourstick}(iv) is trivial regardless of crossing information. No matter the crossing information given to Figure \ref{fig:fourstick}(v)'s diagram, a sequence of Reidemeister I and II moves will eliminate all crossings. Figures \ref{fig:fourstick}(vi) and \ref{fig:fourstick}(vii) represent the trivial class.

\end{proof}

\begin{corollary}

The rail arc classes $2_2$, $2_4$, $2_5$, and $3_2$ each have a stick number of 5.

\end{corollary}

\begin{proof}

This follows from Theorem \ref{thm:lower}, Theorem \ref{thm:4}, and Figure \ref{fig:5stick}.

\end{proof}

\begin{figure}[h]

\includegraphics[scale=.6]{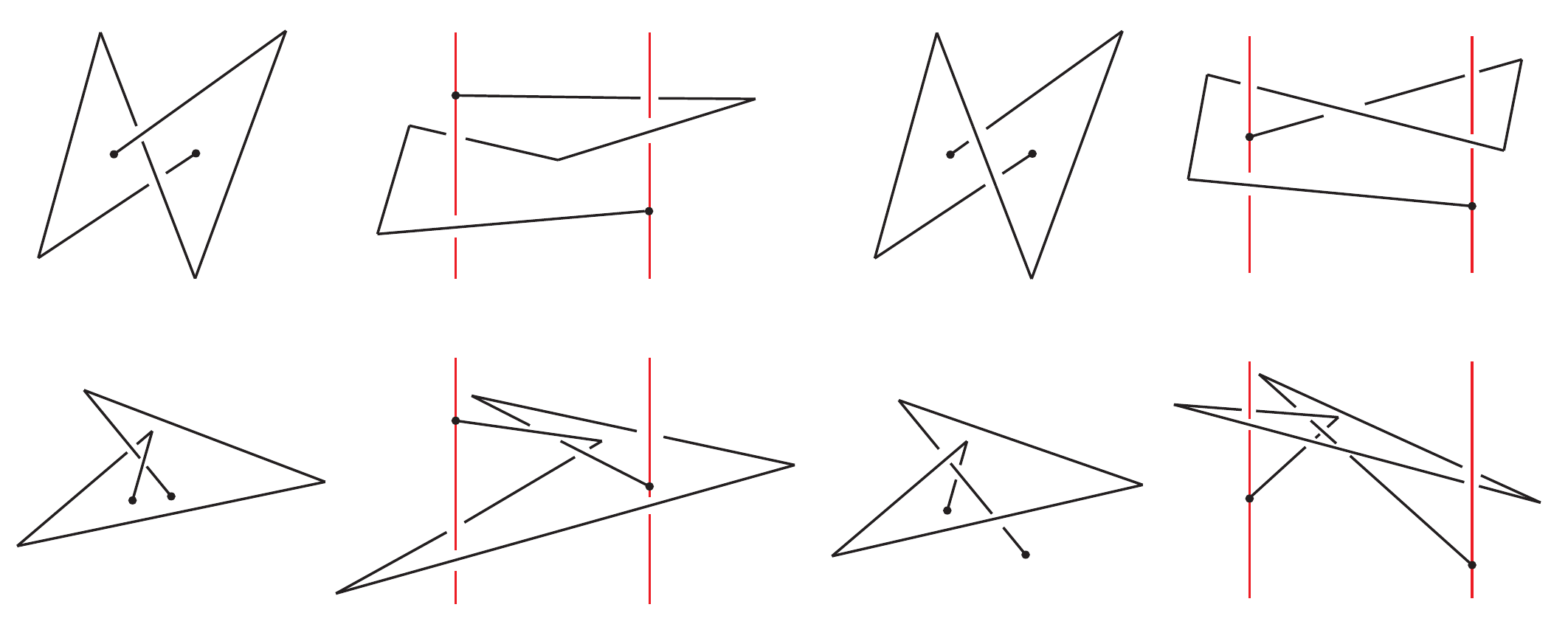}

\caption{Minimal stick representatives of $2_2$, $2_4$, $2_5$, and $3_2$.}
\label{fig:5stick}
\end{figure}

Let $\kappa$ be a planar knotoid diagram with an underpass $\alpha$ embedded in the complement of $\kappa$'s interior. Then, $\gamma:=\kappa\cup\alpha$ is an immersed closed curve in general position. If necessary, canonically smooth where $\kappa$ and $\alpha$ meet so $\gamma$ has well-defined tangent lines at the head and tail of $\kappa$.  For a point $p\in \mathbb{R}^2-\gamma$, suppose that $(r_p(t),\theta_p(t))$ is a parameterization of $\gamma$ in polar coordinates based at $p$ for $0\leq t \leq 1$. The {\it winding number} of $\gamma$ at $p$  is defined by \[ w_\gamma(p)=\frac{\theta_p(1)-\theta_p(0)}{2\pi}\in\mathbb{Z}.\] This can be extended to regular values $p$ of the immersion $\gamma$ with a polar parameterization for $0<t<1$, \[ w_\gamma(p)=\frac{\lim\limits_{t\to1^-}\theta_p(t)-\lim\limits_{t\to0^+}\theta_p(t)}{2\pi}\in\frac{1}{2}\mathbb{Z}.\] 

\begin{lemma}[Kutluay '20 \cite{kutluay2020winding}]

Let $H$ and $T$ be the head and tail of $\kappa$. Then $w_\gamma(H)-w_\gamma(T)=\alpha\cdot \kappa$ where $\alpha\cdot\kappa$ is the intersection number of $\alpha$ and $\kappa$ excluding their endpoints. 
\label{lem}
\end{lemma}

\begin{proof}
As a point $p$ moves along $\alpha$ from $L$ to $H$, the value of $w_\gamma(p)$ increases by 1 when $\alpha$ crosses $\gamma$ from right to left and decreases by 1 when $\alpha$ crosses $\gamma$ from left to right. Since the changes at self-intersections of $\alpha$ cancel each other, the total change in $w_\gamma$ is equal to $\alpha \cdot \kappa$.
\end{proof}

Since $\alpha$ is embedded in the complement of $\kappa$'s interior, Lemma \ref{lem} implies that $w_\gamma(H)=w_\gamma(T)$. As argued in Section 6.2 of \cite{kutluay2020winding}, $w_\gamma(H)$ and $w_\gamma(T)$ remain invariant under Reidemeister moves since these local deformations are taken in a neighborhood excluding the head and tail. Suppose that $\kappa'$ is equivalent to $\kappa$ such that there exists an underpass $\alpha'$ embedded in its complement. Canonically smooth where $\kappa'$ meet $\alpha'$ to get $\gamma'=\kappa'\cup\alpha'$ with well-defined tangent lines at the head and tail of $\gamma'$. There exists a sequence of Reidemeister moves, away from $\kappa'\cap\alpha'$, and planar isotopies taking $\gamma'$ to $\gamma$ while fixing the head and tail since any two under passes of the same knotoid diagram give the same under companion. After the sequence of Reidemeister moves taking $\kappa$ to $\kappa'$, there is a sequence of Reidemeister moves  taking $\alpha$ to $\alpha'$ since they are both under passes.  Therefore, $w_\gamma(H)=w_\gamma(T)=w_\gamma'(H)=w_\gamma'(T)$.

\begin{definition}

Suppose that $[r]$ is a rail arc class with a representative that projects to produce a planar knotoid diagram $\kappa$ such that its head and tail belong to same planar region. Connect the head $H$ and tail $T$ of $\kappa$ with an arc $\alpha$ missing $\kappa$ except at its endpoints such that $\gamma=\kappa \cup \alpha$ is smooth at $\kappa\cap\alpha$. Let \[w[r]:=\text{sgn}(w_\gamma(H))\lfloor|w_\gamma(H)|\rfloor=\text{sgn}(w_\gamma(T))\lfloor|w_\gamma(T)|\rfloor\] be the {\it winding number} of the rail arc class $[r]$ where $\lfloor\cdot\rfloor$ denotes the least integer function and sgn is -1 if $w_\gamma(H)<0$, 1 if $w_\gamma(H)>0$, and 0 if $w_\gamma(H)=0$.

\end{definition}

\begin{remark}
The previous discussion implies that the winding number of rail arc classes, that admit planar knotoid representations with endpoints in the same region, is well-defined and independent of representative. Therefore, two rail arc classes with winding numbers are rail isotopic only if their winding numbers are equal. \end{remark}
 
\begin{remark} Reversing the orientation changes the sign of a rail arc class' winding number. Reversing the orientation of $\gamma=\kappa\cup\alpha$ changes $w_\gamma(H)$ to $-w_\gamma(H)$ since the values $\lim_{t\to1^-}\theta_p(t)$ and $\lim_{t\to0^+} \theta_p(t)$ interchange. When a winding number is well-defined, a rail arc class is equivalent to it orientation reverse only if it has a winding number of zero. 
\end{remark}

\begin{figure}[h]

\includegraphics[scale=.6]{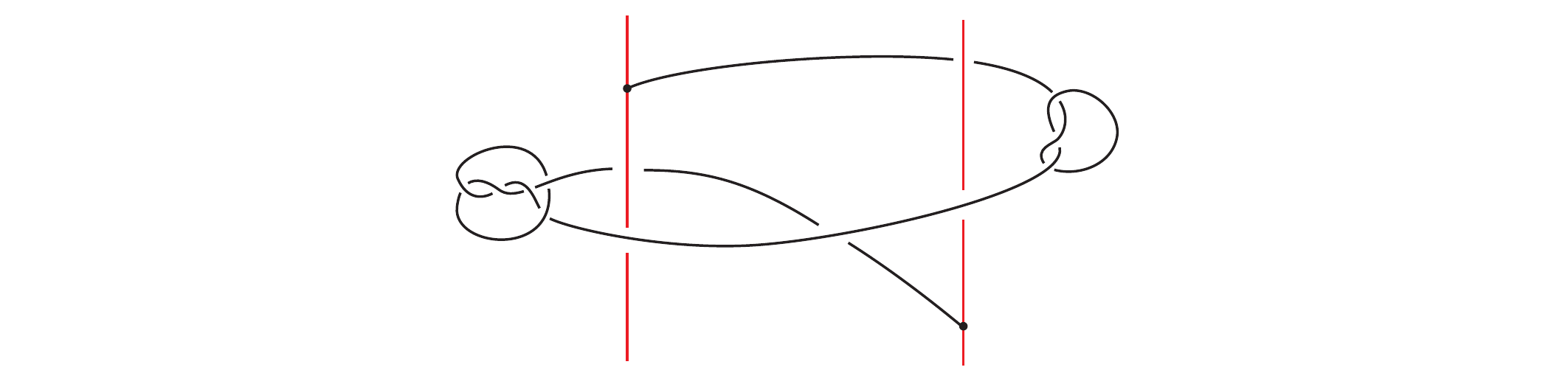}

\caption{Rail arc with a winding number of -2.}
\label{fig:wind}
\end{figure}

\begin{theorem}

For any non-trivial rail arc class $[r]$ that admits a planar knotoid projection $\kappa$ where the head and tail can be connected by a simple arc in the complement of $\kappa$'s interior, the stick number of $[r]$ satisfies \[ 4+2(|w[r]|-1)\leq s[r].\] 

\label{thm:winding}
\end{theorem}

\begin{proof}

\begin{figure}[h]

\begin{overpic}[unit=.43mm,scale=.6]{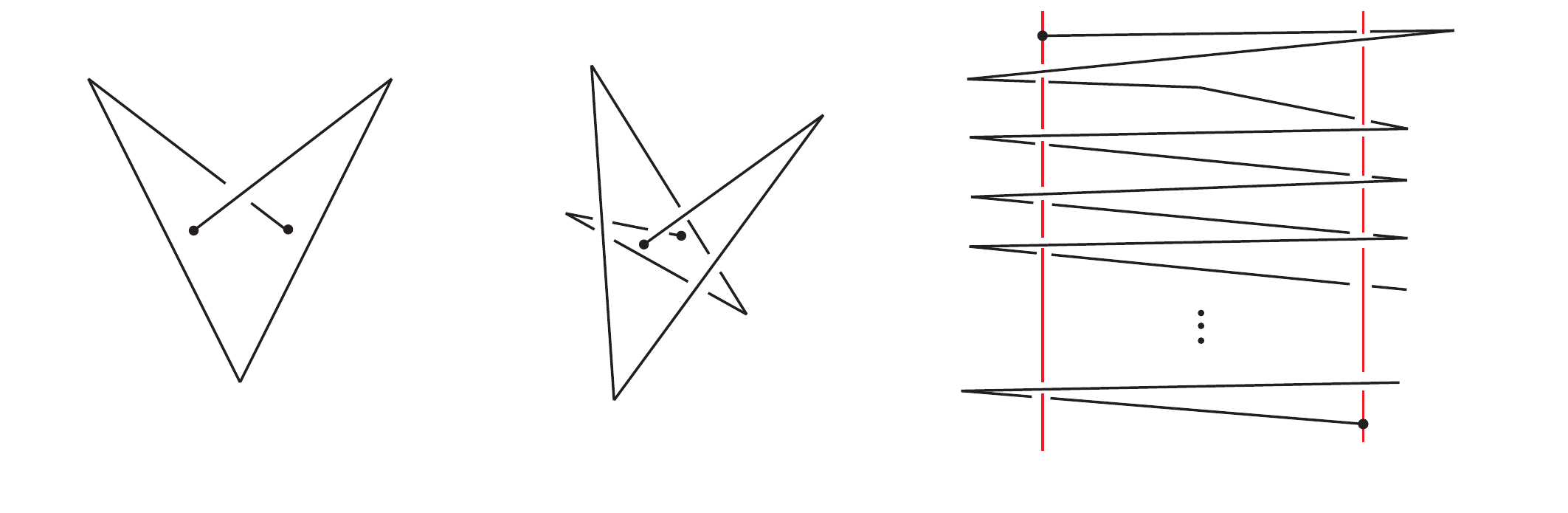}\put(43,2){(i)}\put(123,2){(ii)}\put(226,2){(iii)}

\end{overpic}

\caption{Minimal number of sticks needed to wind.}
\label{fig:minwind}
\end{figure}

In the case that $w[r]=0$, the inequality holds as a consequence of Theorem \ref{thm:lower}. Without loss of generality, assume that the head and tail of $\kappa$ lie on the $x$-axis, the tail is to the left of the head, $w[r]<0$, and the stick stemming from the tail ends in the upper half-plane. This is general since the stick number is involution invariant. In order for the arc to create a non-zero winding number it must pass to the lower half-plane and return to the upper half-plane at least once. At least 2 sticks are needed to go from the upper half-plane to the lower half-plane then back to the upper half-plane. Thus, including the stick connecting to the head, at least 4 sticks are needed to create a non-zero winding number.

 Alternating from the upper and lower half-planes, at least 4 sticks are needed to create a knotoid that winds once around the rails, see Figure \ref{fig:minwind}(i). To create a second winding, the fourth stick can end in the lower half-plane instead of the second rail. At least two additional sticks are needed to complete the second winding, one ending in the upper half-plane and one connecting to the second rail. This minimality argument continues inductively where the final stick of the previous winding is instead taken to end in the lower half-plane and at least two sticks are needed to complete the next winding, see Figures \ref{fig:minwind}(ii), \ref{fig:minwind}(iii).

\end{proof}

Let $[W_n]$ denote the class of the {\it winding rail arc} given by the representative $W_n$ that winds monotonically around and down its rails $n\in\mathbb{Z}$ times clockwise if $n<0$ and counter-clockwise if $n>0$, see Figure \ref{fig:w_n} for a representative  with $n<0$. Connecting the endpoints of the planar knotoid diagram $W_n$ with a line segment $\alpha$ shows that $w[W_n]=n$ since the angle parametrization of $\gamma=W_n\cup\alpha$ at the head $H$ increases or decreases depending if $n$ is positive or negative. This gives infinitely many rail arc classes that are not equivalent to their rev($\cdot$), sym($\cdot$), or rot($\cdot$) since these involutions invert the winding number. Reversing the orientation negates the winding number. Reflecting the planar knotoid diagram across a vertical line to produce sym($\cdot$) and rot($\cdot$) also negates the winding number.

\begin{figure}[h]

\includegraphics[scale=.6]{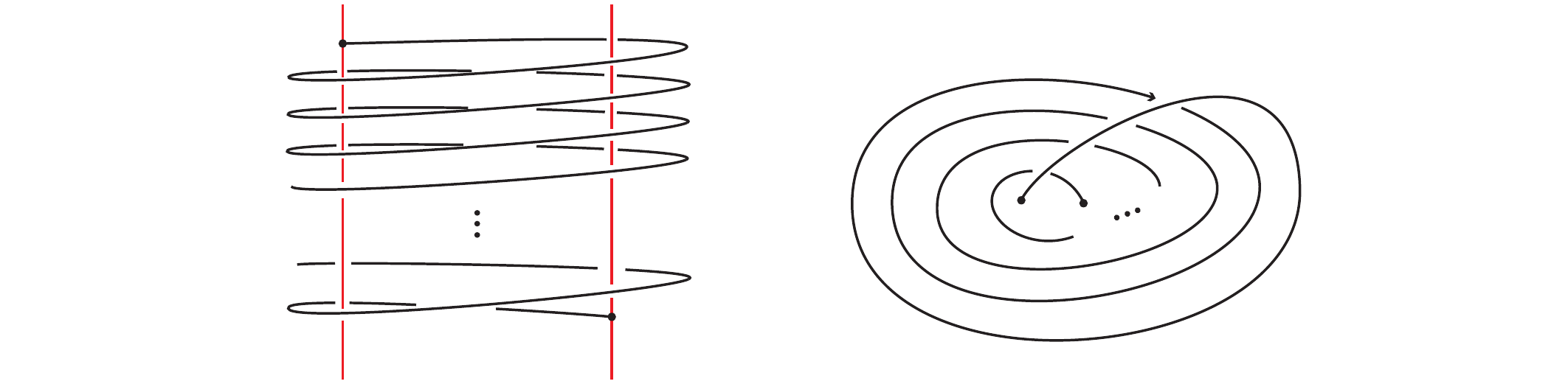}

\caption{Rail arc diagram and planar knotoid diagram of the winding rail arc $[W_n]$.}
\label{fig:w_n}
\end{figure}

\begin{corollary}

The winding rail arc class $[W_n]$ has a stick number of $4+2(|n|-1)$ for any non-zero integer $n$.

\label{cor:winding}

\end{corollary}

\begin{proof}

This follows from Theorem \ref{thm:winding}  and Figure \ref{fig:minwind}(iii).

\end{proof}

\begin{corollary}

Both $2_3$ and $2_6$ have a stick number of 6.

\end{corollary}

\begin{proof}

Since $w[2_3]=w[2_6]=-2$, this follows from Theorem \ref{thm:winding} and Figures \ref{fig:minwind}(ii) \& \ref{fig:minstick}.

\end{proof}

\begin{figure}[h]

\includegraphics[scale=.6]{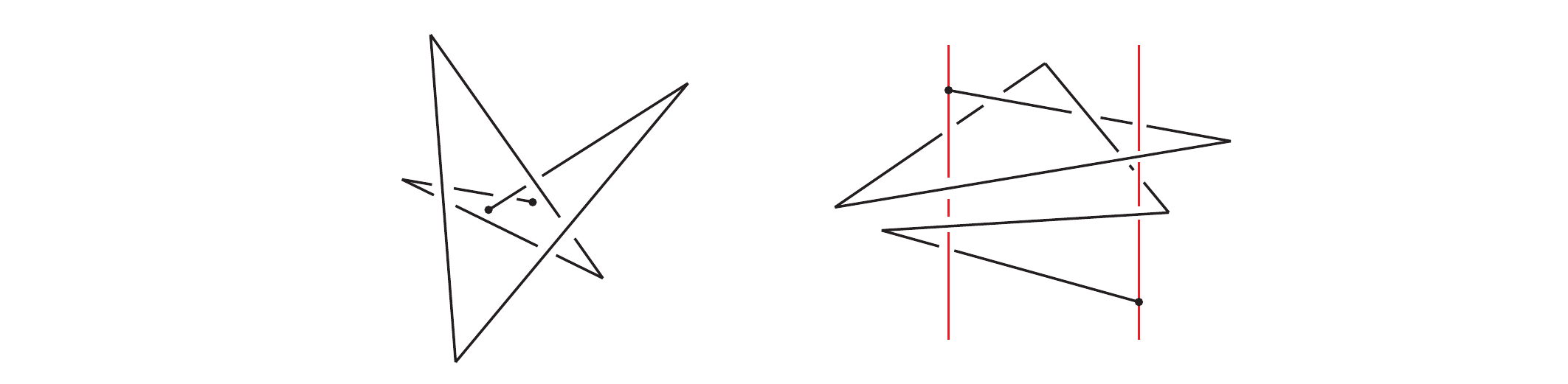}

\caption{Minimal stick representative of $2_6$.}
\label{fig:minstick}
\end{figure}

This completes the stick number calculation of all rail arc classes with a crossing number at most 2.

\section{The Rail stick number of Knots and Links}
\label{railindex}

This section relates the stick number of knots to the stick number of rail arcs.

\begin{definition}

The {\it rail stick number} of a p.l. knot class $[K]$ is the minimum number of sticks needed to create a stick rail arc whose under or over companion is $[K]$. The rail stick number of a knot class $[K]$ is denoted $rs[K]$.

\end{definition}

 There are many non-rail isotopic rail arcs that have isotopic companions. The method used to calculate the rail stick number of knots, Theorem \ref{thm:companion}, will also calculate the stick number of  the rail arc classes used. Calculating the stick numbers of all rail arc classes with small crossing number is a more daunting task than calculating all rail stick numbers of knots with  small crossing number since there are 1137 prime rail arc  classes with crossing number at most five \cite{goundaroulis2019systematic}. The companion map taking rail arc classes to their under (or over) companion is surjective, remove a small arc of a knot diagram to produce a planar knotoid that maps to the knot class, but is emphatically not injective.

 \begin{theorem}

 For any knot class $[K]$, \[  s[K]-2 \leq rs[K] \leq s[K]-1.\]
 
 \label{thm:companion}
 
 \end{theorem}

 \begin{proof}
 
 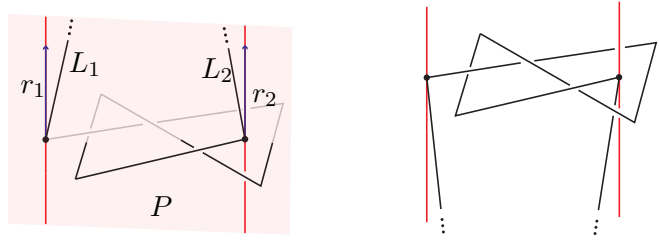
\begin{figure}[h]

\begin{overpic}[unit=.4mm,scale=.6]{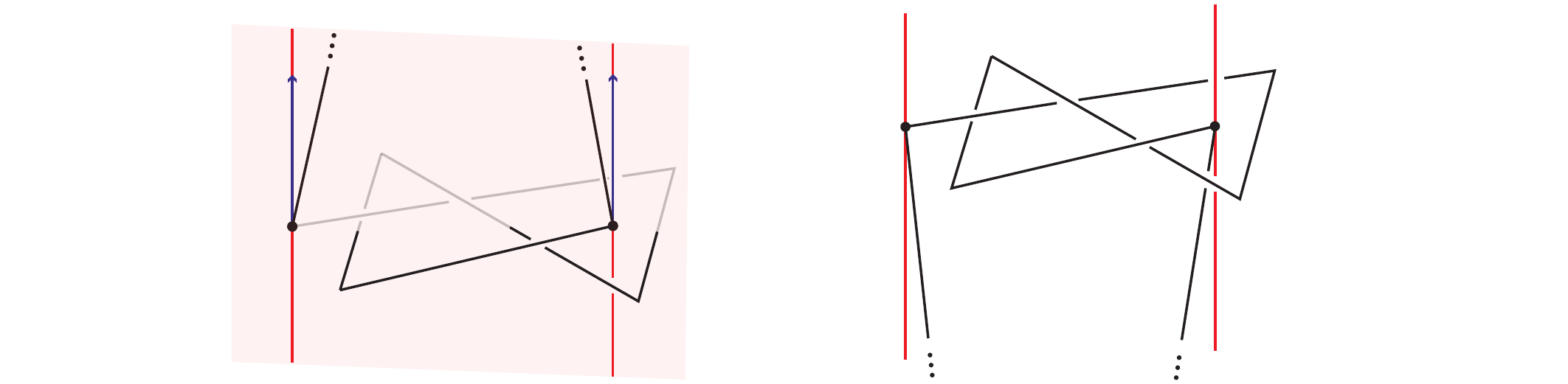}\put(95,7){$P$}\put(52,48){$r_1$}\put(129,45){$r_2$}\put(68,55){$L_1$}\put(112,53){$L_2$}

\end{overpic}

\caption{Two stick over and under passes.}
\label{fig:companion}
\end{figure}

Let $r$ be a stick rail arc. Consider the plane $P$ containing both rails. Let $r_1$ and $r_2$ denote the rays that begin at the head and tail of $r$ and run along the rails in the positive direction. Create an overpass by first adding a stick $L_1\subset P$ starting at the head. Further suppose that $r_1$ and $L_1$ meet at an arbitrarily small angle in $P$ and $L_1$ skews in the direction of the other rail. Include another stick $L_2\subset P$ starting at the tail of the rail arc such that its angle with $r_2$ in $P$ is arbitrarily small and $L_2$ skews in the direction of the other rail. Since $L_1$ and $L_2$ are not parallel in $P$ they will intersect if made sufficiently large. Take $L_1$ and $L_2$ to have a length such that they meet at a point above the rail arc. With their angles against the rails arbitrarily small, the union of $L_1$ and $L_2$ is an overpass, see Figure \ref{fig:companion}. A similar argument shows that there exists a two stick underpass.

Suppose that $r$ realizes the rail stick number of its over (or under) companion. If the stick number of $r$ were to be less than $s[K]-2$, then the addition of the previously constructed two stick overpass (or underpass) creates a companion stick knot with fewer sticks than its stick number. This contradiction implies that $s[K]-2\leq rs[K]$.

 Now, let $K$ be a stick knot realizing the stuck number of its class $[K]$. Consider a projection of $K$ such that a stick $\mathcal{N} \subset K$ is projected to a point. Take a small perturbation of this projection to get a knot diagram with $\mathcal{N}$ barely visible. Since this perturbation can be taken arbitrarily close to the projection eliminating $\mathcal{N}$, it can be assumed that $\mathcal{N}$ contains no crossings in the diagram. Remove $\mathcal{N}$ to get a planar knotoid diagram that lifts to a rail arc whose under and over companion is $[K]$. This stick rail arc has one less stick than the stick number of $[K]$, i.e. $rs[K] \leq s[K]-1.$

 \end{proof}

 \begin{theorem}
 
 The 29 non-starred knot classes in Table \ref{tab:sticknumber} satisfy,
 \[ rs[K]=s[K]-2.\]
 
 \label{thm:knots}
 \end{theorem}
 
 \begin{proof}

Let $K$ be a stick knot realizing the stick number of its class $[K]$. Suppose that $K$ admits a knot diagram with 2 sticks projecting underneath all other sticks that they may intersect. Removing these two sticks produces a planar knotoid diagram that lifts to a rail arc whose under companion is $[K]$. 

Known stick numbers of small crossing knots are given in Table \ref{tab:sticknumber}. This table is from the appendix of \cite{eddy}.  Scharein created 3D models of low crossing stick knots using his \texttt{KnotPlot} software found at \cite{knotplot}. The planar knotoid diagrams of Figures \ref{fig:table1} and \ref{fig:table2}  are created by projecting minimal stick representatives and removing 2 sticks that projected under all intersecting sticks. Therefore, these planar knotoid diagrams lift to stick rail arcs with 2 fewer sticks than their under companions' stick numbers. The result follows from the lower bound of Theorem \ref{thm:companion}. Each planar knotoid diagram is labeled with its under companion's calculated rail stick number. 
 
 \begin{remark}
 As noted by Clayton Shonkwiler, not all of Scharein's purportedly minimal stick knots actually realize their class' stick number, namely the starred knot classes of Table \ref{tab:sticknumber}. Although, the knot classes mentioned in the proof of Theorem \ref{thm:knots} are minimal in the 3D model database of Scharein. 
 
 \label{rmk:clayton}
 \end{remark}

\begin{figure}

\begin{overpic}[unit=.428mm,scale=.6]{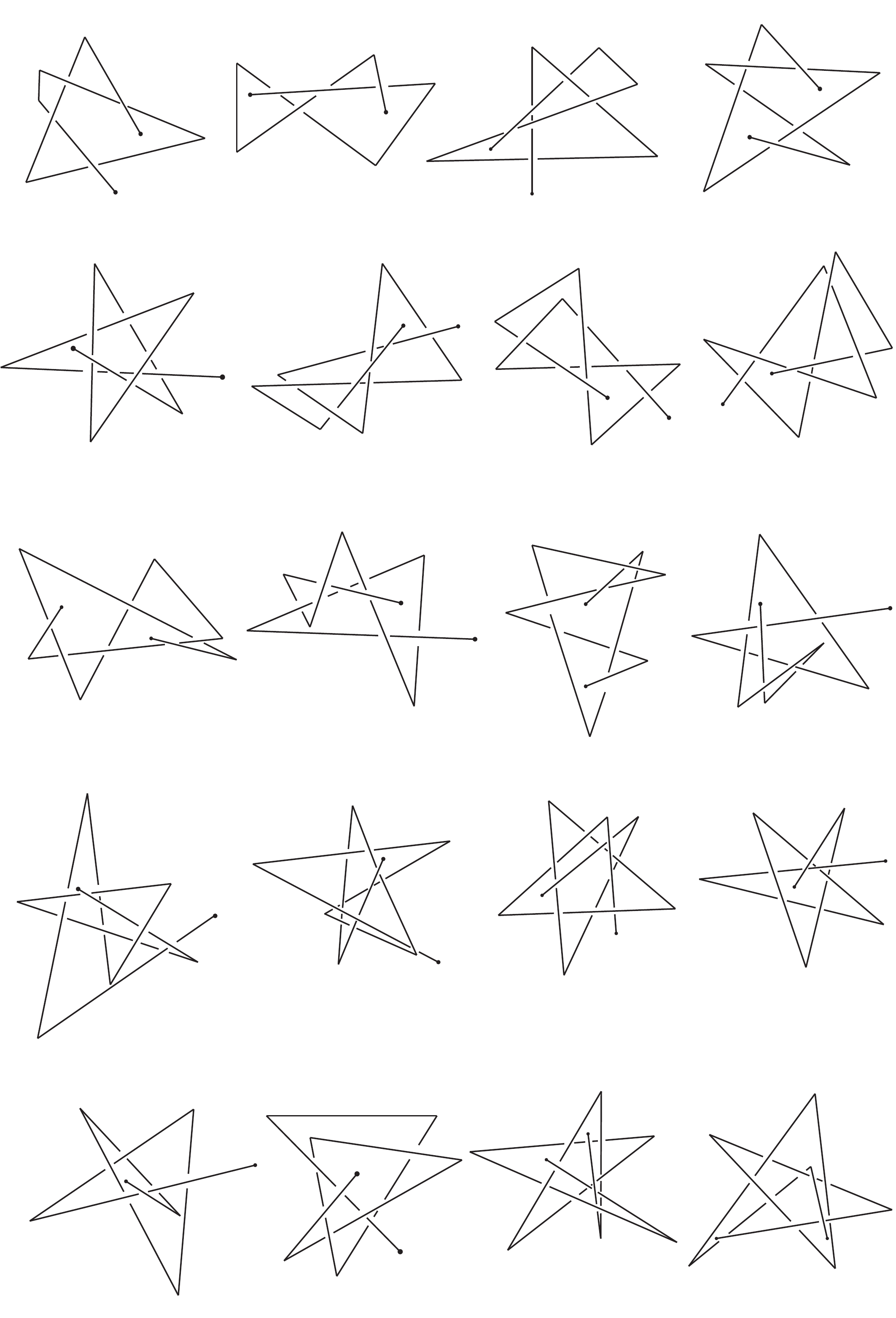}\put(19,438){$rs[5_1]=6$}\put(95,438){$rs[5_2]=6$}\put(174,438){$rs[6_1]=6$}\put(260,438){$rs[6_2]=6$}

\put(19,361){$rs[6_3]=6$}\put(95,361){$rs[7_1]=7$}\put(174,361){$rs[7_2]=7$}\put(260,361){$rs[7_3]=7$}

\put(19,269){$rs[7_4]=7$}\put(95,269){$rs[7_5]=7$}\put(174,269){$rs[7_6]=7$}\put(260,269){$rs[7_7]=7$}

\put(19,181){$rs[8_{16}]=7$}\put(95,181){$rs[8_{17}]=7$}\put(174,181){$rs[8_{18}]=7$}\put(260,181){$rs[8_{19}]=6$}

\put(19,85){$rs[8_{20}]=6$}\put(95,85){$rs[8_{21}]=7$}\put(174,85){$rs[9_{29}]=7$}\put(260,85){$rs[9_{34}]=7$}

\end{overpic}

\caption{Minimal stick representatives realizing their labeled under companion's rail stick number, 1 of 2.}
\label{fig:table1}
\end{figure}

\begin{figure}

\begin{overpic}[unit=.428mm,scale=.6]{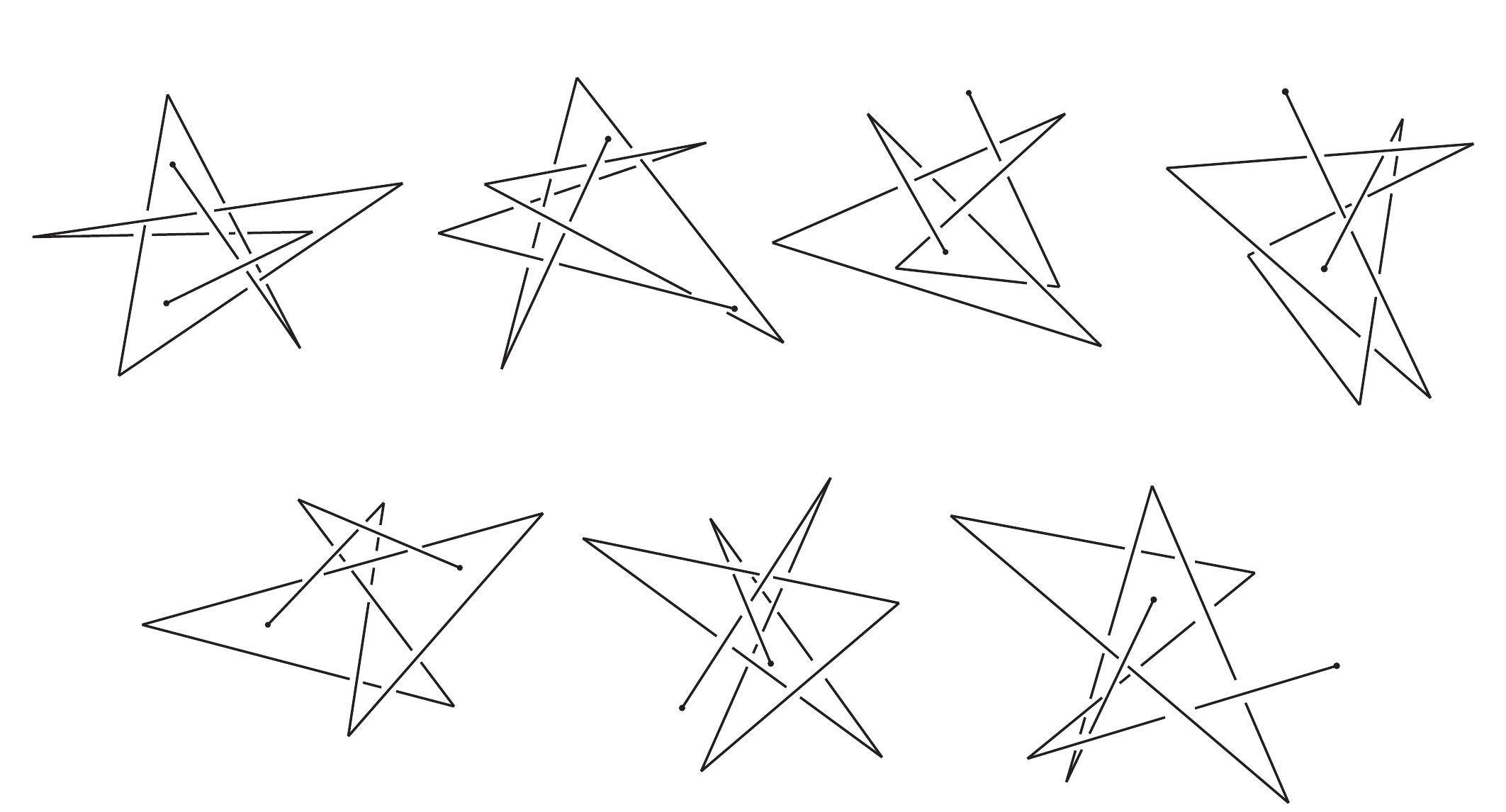}\put(23,152){$rs[9_{40}]=7$}\put(103,152){$rs[9_{41}]=7$}\put(175,152){$rs[8_{42}]=7$}\put(248,152){$rs[9_{44}]=7$}

\put(56,72){$rs[9_{46}]=7$}\put(134,72){$rs[9_{47}]=7$}\put(215,72){$rs[9_{49}]=7$}

\end{overpic}

\caption{Minimal stick representatives realizing their labeled under companion's rail stick number, 2 of 2.}
\label{fig:table2}
\end{figure}

 \end{proof}

\begin{table}
\setlength{\tabcolsep}{10pt} 
\renewcommand{\arraystretch}{1.5}
\raggedright
\begin{tabular}{|c||c |c |c |c| c| c |c |c |c| c| c| c| c| c| c| c| c| c| c| c| c| c|c|c|c|c|c|c|}

\hline 
$[K]$ & $3_1$ & $4_1$ & $5_1$ & $5_2$ & $6_1$ & $6_2$ & $6_3$ & $7_1$ &  $7_2$& $7_3$  \\ \hline 
$s[K]$ & $6$ & $7$ & $8$ & $8$ & $8$ & $8$ & $8$ & $9$ &  $9$& $9$  \\ \hline

\end{tabular}

\vspace{5mm}
\raggedright

\begin{tabular}{|c||c |c |c |c| c| c |c |c |c| c| c| c| c| c| c| c| c| c| c| c| c| c|c|c|c|c|}

\hline 
$[K]$ &$7_4$ &$7_5$& $7_6$ & $7_7$ & $8_{16}$ & $8_{17}$ & $8_{18}$ & $8_{19}$ & $8_{20}$  \\ \hline
$s[K]$ &$9$ &$9$& $9$ & $9$  
& 9&  $9$ & $9$ & $8$ & $8$  \\ \hline

\end{tabular}

\vspace{5mm}
\raggedright

\begin{tabular}{|c||c |c |c |c| c| c |c |c |c| c| c| c| c| c| c| c| c|c|c|c|c|}

\hline 
$[K]$ & $8_{21}$ & $9_{29}$ & $9_{34}$ & $9_{35}^*$ &  $9_{39}^*$& $9_{40}$ &$9_{41}$ &$9_{42}$    \\ \hline
$s[K]$ & $9$ & $9$ & $9$ & $9$ &  $9$& $9$ &$9$ &$9$  \\ \hline

\end{tabular}
\vspace{5mm}

\raggedright

\begin{tabular}{|c||c |c |c |c| c| c |c |c |c| c| c| c| c| c| c| c| c|c|c|c|c|}

\hline 
$[K]$  & $9_{43}^*$ & $9_{44}$ & $9_{45}^*$ & $9_{46}$ & $9_{47}$ & $9_{49}$  \\ \hline
$s[K]$ & $9$ & $9$ & $9$ & $9$ & $9$ & $9$   \\ \hline

\end{tabular}
\vspace{5mm}

\caption{Stick numbers of some small crossing knots \cite{eddy}.}\label{tab:sticknumber}
\end{table}

Turaev also introduced knotoids with additional circle components in his seminal publication  \cite{turaev2012knotoids}.

\begin{definition}

A {\it planar multi-knotoid diagram} is a generic immersion of the unit interval and several copies of $S^1$ in the plane endowed with crossing information. Reidemeister moves away from the boundary points and planar isotopy generate an equivalence relation on planar multi-knotoid diagrams with classes called {\it planar multi-knotoids}. The {\it under (over) companion} of a planar multi-knotoid is the link class defined by taking a planar multi-knotoid diagram representative and connecting the head and tail by a generic simple arc that is given under  (over) crossing information whenever it meets the diagram.

\end{definition}

\begin{definition}

A rail arc with a link embedded in the complement of the rails and arc is called a {\it multi-component rail arc}. Two multi-component rail arcs are equivalent if there exists a rail isotopy taking one to the other. The equivalence classes are called {\it rail isotopy classes}. 

\end{definition}

Planar multi-knotoids uniquely lift to rail isotopy classes of multi-component rail arcs, and rail isotopy classes of multi-component rail arcs uniquely project to planar multi-knotoids in a plane perpendicular to the rails. Therefore, there is a correspondence between planar multi-knotoids under the relation of Reidemeister moves away from the tail and head, and rail isotopy classes of multi-component rail arcs.

\begin{definition}

A {\it multi-component stick rail arc} is a stick rail arc with additional stick knot components in the complement of the arcs and rails. 

\end{definition}

\begin{definition}

The {\it stick number} of a multi-component rail arc class is the minimum number of sticks needed to create a p.l. representative.

\end{definition}

The rail stick number of knots is now generalized to links.

\begin{definition}

The {\it rail stick number} of a p.l. link class $[L]$ is the minimum number of sticks needed to create a multi-component stick rail arcs whose under or over companion is $[L]$.

\end{definition}

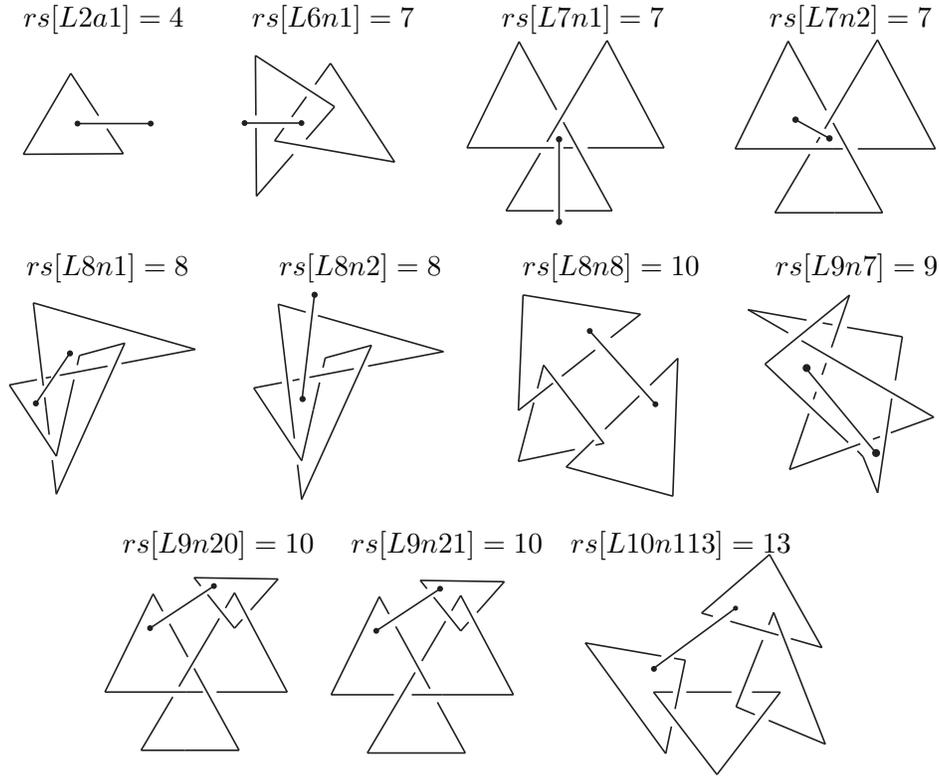
\begin{figure}[h]

\begin{overpic}[unit=.4mm,scale=.6]{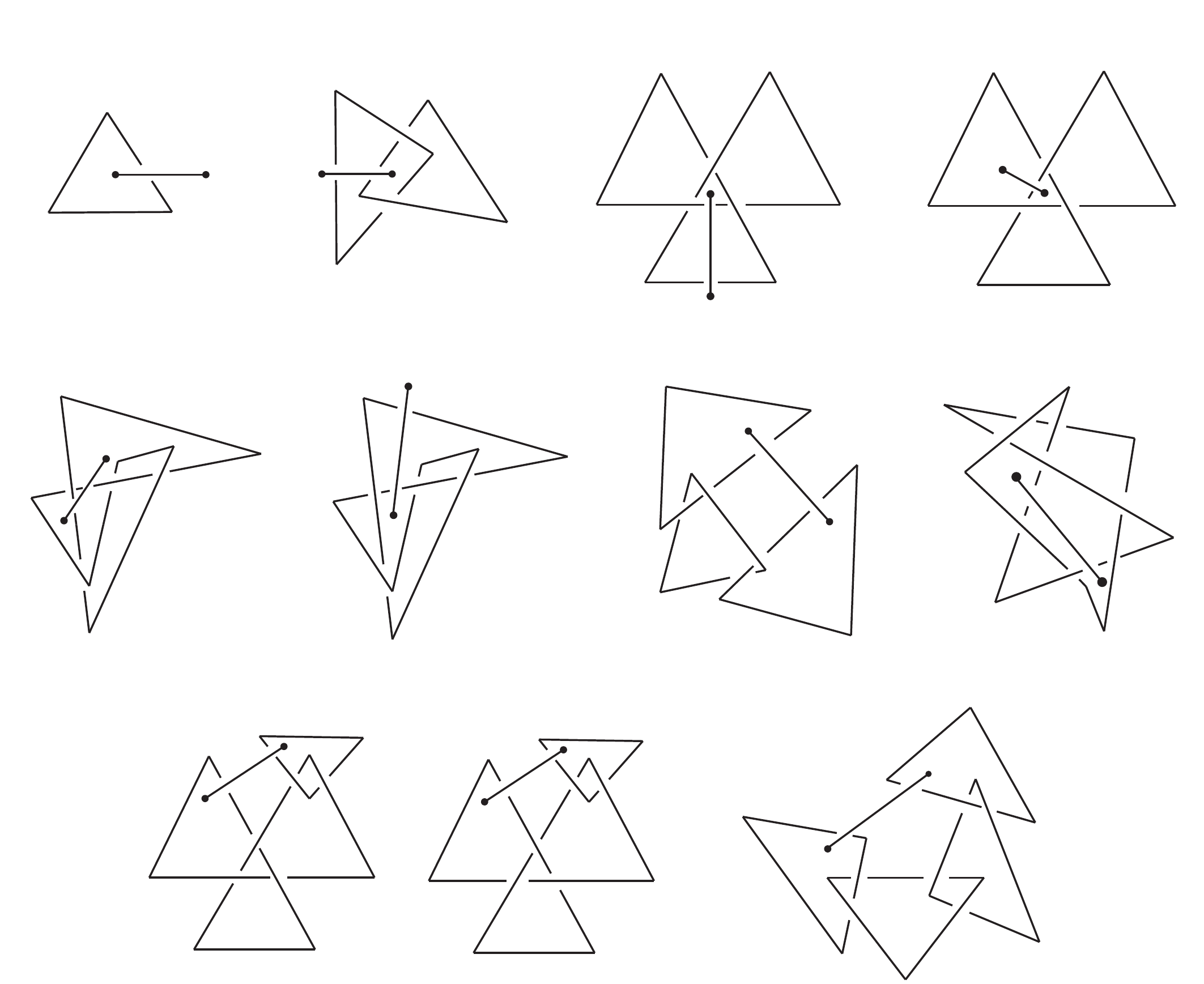}\put(13,250){$rs[L2a1]=4$}\put(89,250){$rs[L6n1]=7$}\put(172,250){$rs[L7n1]=7$}\put(261,250){$rs[L7n2]=7$}

\put(14,167){$rs[L8n1]=8$}\put(98,167){$rs[L8n2]=8$}\put(179,167){$rs[L8n8]=10$}\put(263,167){$rs[L9n7]=9$}

\put(46,75){$rs[L9n20]=10$}\put(122,75){$rs[L9n21]=10$}\put(195,75){$rs[L10n113]=13$}

\end{overpic}

\caption{Multi-knotoids in the plane that lift to multi-component stick rail arcs realizing their stick numbers and the rail stick numbers of their under companions.}\label{fig:multistick}
\end{figure}

\begin{table}[h]
\setlength{\tabcolsep}{10pt} 
\renewcommand{\arraystretch}{1.5}
\raggedright
\begin{tabular}{|c||c |c |c |c| c| c |c |c |c| c| c| c| c| c| c| c| c| c| c| c| c| }

\hline 
$[L]$ & $L2a1$ & $L6n1$ & $L7n1$ & $L7n2$ & $L8n1$ & $L8n2$ & $L8n8$ & $L9n7$     \\ \hline

$rs[L]$ & 4& 7& 7& 7& 8& 8& 10& 9 

 \\ \hline

$s[L]$ & 6& 9& 9& 9& 10& 10& 12& 11 

\\\hline

\end{tabular}

\vspace{5mm}
\raggedright

\begin{tabular}{|c||c |c |c |}

\hline 
$[L]$ &  $L9n20$ & $L9n21$& $L10n113$    \\ \hline

$rs[L]$ & 10& 10& 13 \\\hline

$s[L]$ & 12& 12& 15 \\\hline

\end{tabular}
\vspace{5mm}

\caption{Rail stick numbers of several links.}\label{tab:links}
\end{table}

\begin{theorem}

Figure \ref{fig:multistick} shows 11 multi-component stick rail arcs that realize their stick numbers and Table \ref{tab:links} gives the rail stick numbers of their under companions.

\end{theorem}

\begin{proof}

Suppose that $\mu$ is a multi-component stick rail arc with  rail arc component $r$ and knot components $K_i$ for $i=1,\dots, n$. Then,  \begin{equation} s[r] + \sum_{i=1}^n s[K_i]\leq s[\mu].  \label{stick1}\end{equation} In each of the 11 multi-component stick rail arcs of Figure \ref{fig:multistick}, the rail arc component and each knot component realizes its class' stick number. Each rail arc component is trivial and made with a single stick. The knot components are the unknot with 3 sticks, the trefoil with 6 sticks, the figure eight with 7 sticks, and the $5_2$ with 8 sticks. Since no component could be made with fewer sticks, the inequality (\ref{stick1}) is sharp on these multi-component stick rail arc classes.

 Suppose further that $[L]$ is a  companion of $\mu$.  From the proof of Theorem \ref{thm:companion}, there is a two stick underpass (or overpass) that turns $\mu$ into a stick knot representative of $[L]$. Therefore, \begin{equation} s[L]-2 \leq rs[L]\leq s[\mu]. \label{multirail}\end{equation} Since each multi-component stick rail arcs that the planar multi-knotoid diagrams in Figure \ref{fig:multistick} lift to have two fewer sticks than its under companion's stick number, labeled in the figure, the inequality (\ref{multirail}) is an equality on all 11 link classes shown in Table \ref{tab:links}.
\end{proof}

Figure \ref{fig:multifamily2} gives an infinite family of multi-component stick rail arcs that realizes the rail stick numbers of their under companions.

\begin{figure}[h]

\includegraphics[scale=.6]{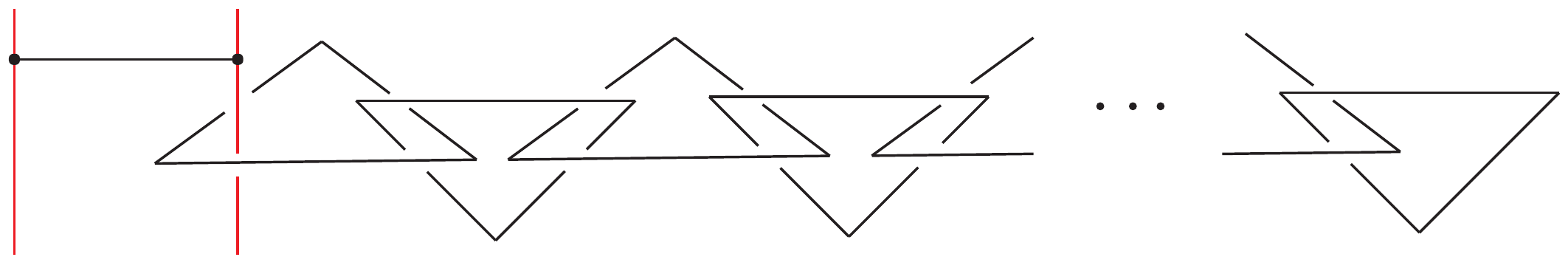}

\caption{Multi-component rail arc with $n$ unknot components and a stick number of $3n+1$. }
\label{fig:multifamily2}
\end{figure}

\section{Rail Arcs in the Simple Cubic Lattice}
\label{cubic}

The definitions in the previous section can be reformulated in the cubic lattice.

\begin{definition}

The {\it rail lattice stick number} of a knot class $[K]$ is the minimum number of sticks needed to create a lattice rail arc whose under or over companion is $[K]$. The rail lattice stick number is denoted $rs_{CL}[K]$.

\end{definition}

\begin{theorem}

For any knot $K$, \[ s_{CL}[K]-4\leq rs_{CL}[K].\]

\label{thm:cubic}
\end{theorem}

\begin{figure}[h]

\includegraphics[scale=.6]{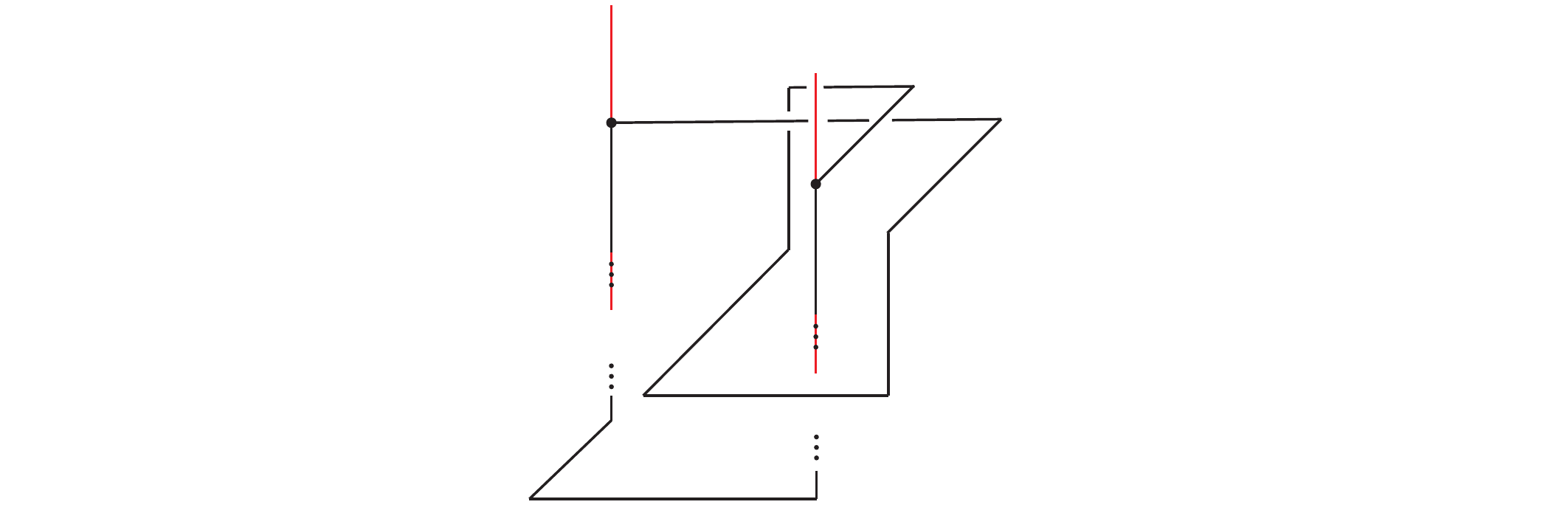}

\caption{Maximal stick underpass of a lattice rail arc.}
\label{fig:cubiccompanion}
\end{figure}

\begin{proof}

For any lattice rail arc $r$, an under companion can be created by first adding two sticks $L_1$ and $L_2$ starting at the tail and the head of $r$ and transversing arbitrarily far down the rails. Let these sticks end in the same plane perpendicular to the rails. In this plane, at most two sticks are needed to connect the endpoints of $L_1$ and $L_2$. This creates an underpass comprised of four sticks, see Figure \ref{fig:cubiccompanion}. Likewise, an overpass comprised of four sticks can be created. Therefore, $r$ cannot have fewer sticks than four less the lattice stick numbers of its companions.

\end{proof}

\begin{remark}

The inequality of Theorem \ref{thm:cubic} is not sharp on all knot classes. The rail lattice stick number of the unknot $[U]$ is 1 while $s_{CL}[U]-4=0$.

\end{remark}

\begin{corollary}

The rail lattice stick number of $3_1$ is 8, $rs_{CL}[4_1]=10$, \[rs_{CL}[8_{20}]=rs_{CL}[8_{21}]=rs_{CL}[9_{46}]=14,\] and $rs_{CL}[\ \text{$(p,p+1)$-torus knot} \ ]=6p-4$ for any integer $p\geq 2$. The lattice stick number of the rail arc classes $2_1$ and $2_4$ are 8 and 10.

\end{corollary}

\begin{proof}

The rail lattice stick numbers follow from Theorem \ref{thm:cubic} and Figures \ref{fig:cubiccompanion}, \ref{fig:fig8cubic}, \ref{fig:8_20} \& \ref{fig:torus}. For any rail arc class $[r]$ with companion $[K]$,  \[ rs_{CL}[K]  \leq s_{CL}[r].\] Thus, the lattice stick numbers of $2_1$ and $2_4$ follow from this inequality and the representatives in Figures \ref{fig:cubiccompanion} and \ref{fig:fig8cubic}.

\end{proof}

\begin{figure}[h]

\includegraphics[scale=.6]{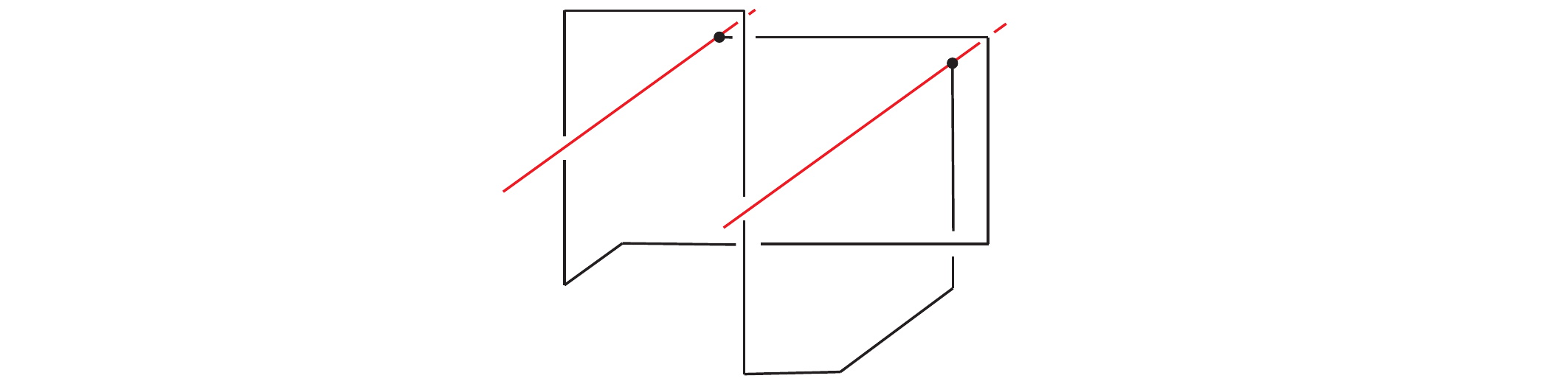}

\caption{Minimal stick representative of a lattice rail arc whose over companion is $4_1$, see \cite{Huh_2010}.}
\label{fig:fig8cubic}
\end{figure}

\begin{figure}[h]

\includegraphics[scale=.6]{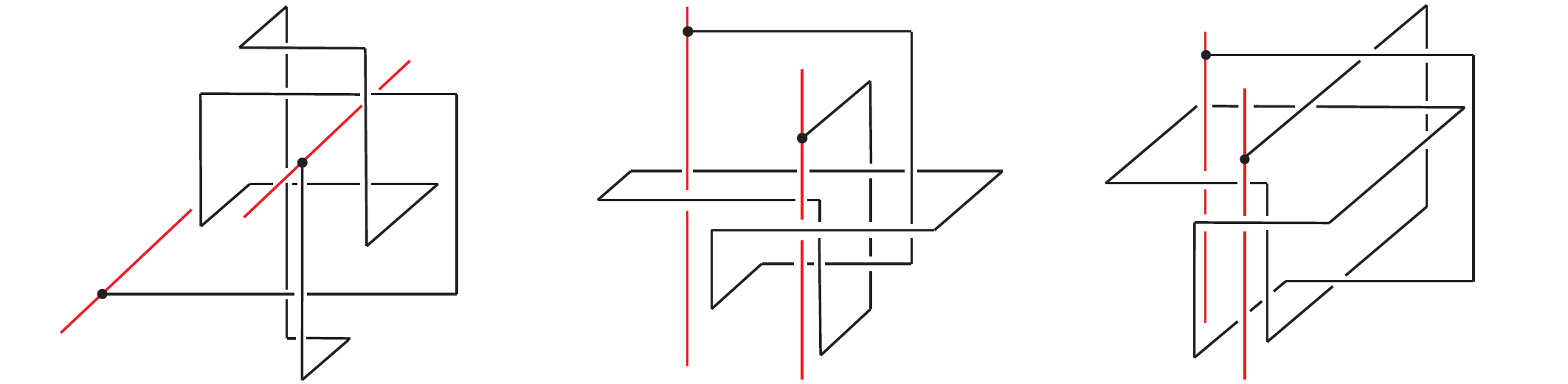}

\caption{Minimal stick representatives of lattice rail arcs whose under companions are $8_{20}$, $8_{21}$, and $9_{46}$, see \cite{adams}.  }

\label{fig:8_20}
\end{figure}

\begin{figure}[h]

\includegraphics[scale=.6]{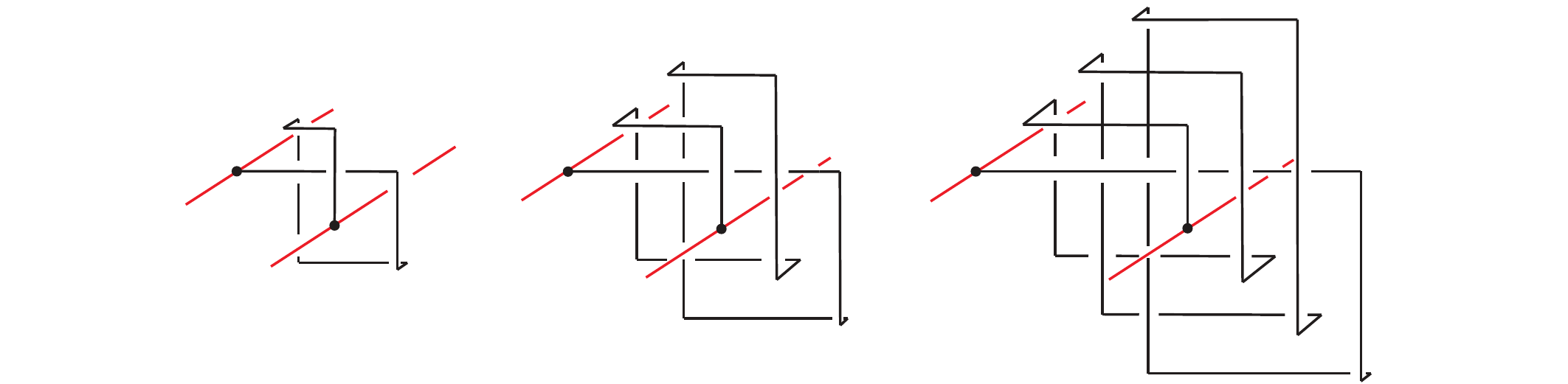}

\caption{Minimal stick representatives of lattice rail arcs whose under companions are the $(p,p+1)$-torus knots, see \cite{adams}.}
\label{fig:torus}
\end{figure}

\begin{theorem}

Figure \ref{fig:cubiclinks} shows 6 multi-component lattice rail arcs that realize their lattice stick numbers and Table \ref{tab:cubiclinks} gives the rail lattice stick numbers of their under companions.

\end{theorem}

\begin{proof}

Suppose that $\mu$ is a multi-component lattice rail arc with  rail arc component $r$ and knot components $K_i$ for $i=1,\dots, n$. Then,  \begin{equation} s_{CL}[r] + \sum_{i=1}^n s_{CL}[K_i]\leq s_{CL}[\mu].  \label{stick2}\end{equation} The rail arc and each knot component of the 6 multi-component rail arcs in Figure \ref{fig:multistick} realizes its class' lattice stick number. Therefore, the inequality (\ref{stick2}) is sharp on these multi-component lattice rail arcs.

 Suppose further that $[L]$ is a  companion of $\mu$.  From the proof of Theorem \ref{thm:cubic}, there is a 4 stick underpass (or overpass) that turns $\mu$ into a lattice knot representative of $[L]$. Therefore, \begin{equation} s_{CL}[L]-4 \leq rs_{CL}[L]\leq s_{CL}[\mu]. \label{multirail}\end{equation}

\noindent If $[L]$ has $n$ components, each unknotted, then $4n-3\leq rs_{CL}[L]$.

\end{proof}

\begin{figure}[h]

\includegraphics[scale=.6]{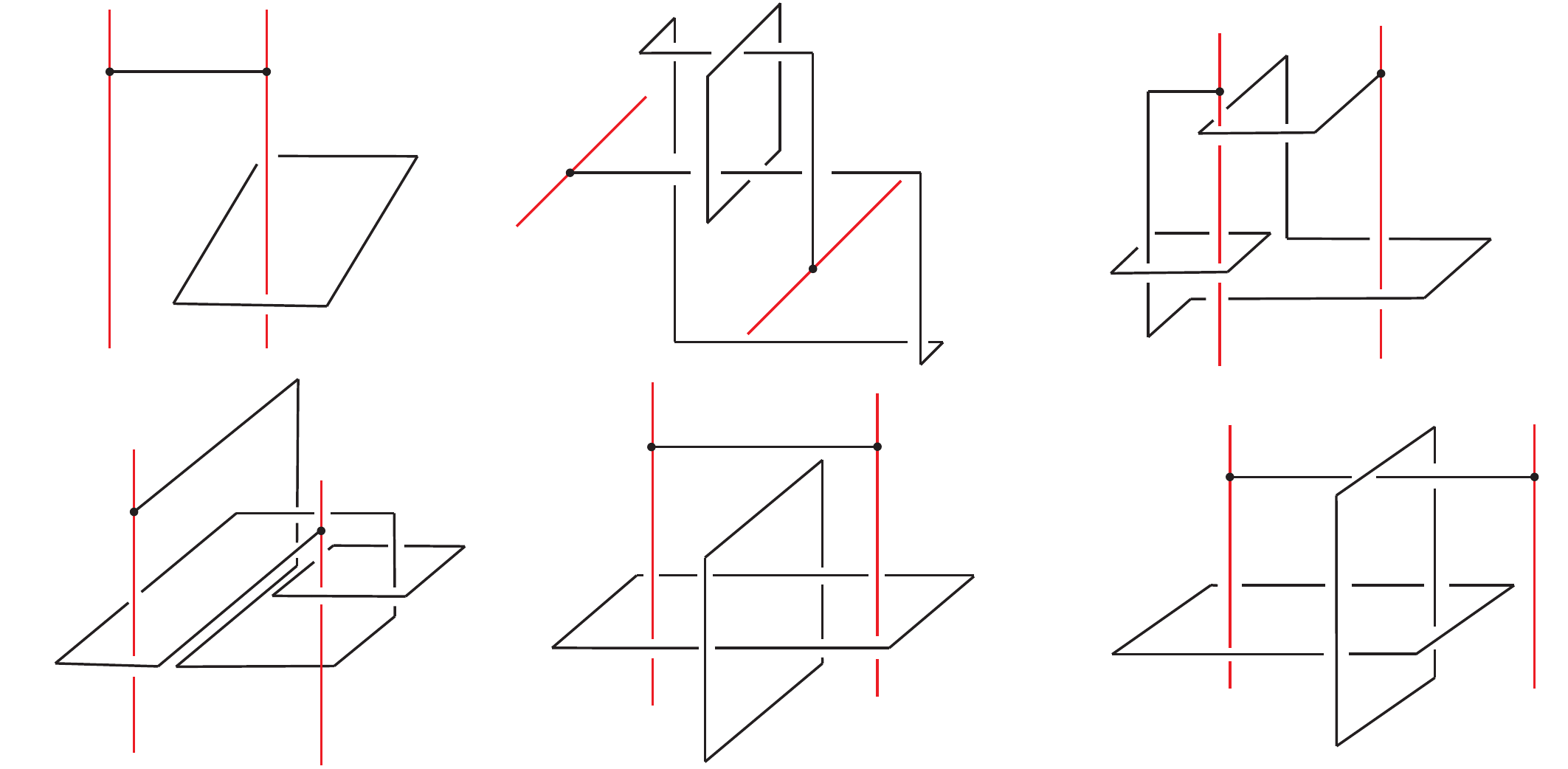}

\caption{Minimal stick representatives of multi-component lattice rail arcs, see \cite{adams}.}
\label{fig:cubiclinks}
\end{figure}

\begin{table}[h]
\setlength{\tabcolsep}{10pt} 
\renewcommand{\arraystretch}{1.5}
 
\begin{tabular}{|c||c |c |c |c| c| c |c |c |c| c| c| c| c| c| c| c| c| c|}

\hline 
$L$ & $L2a1$ & $L7n1$ & $L8n2$ & $L8a21$& $L6a4$ & $L6n1$    \\ \hline

$rs_{CL}[L]$ & 5& 12& 14& 14& 9 &9 \\\hline

$s_{CL}[L]$ & 8& 16& 18& 18& 12 &12 \\\hline

\end{tabular}
\vspace{5mm}

\caption{Rail lattice stick numbers of several links.}\label{tab:cubiclinks}
\end{table}

This paper concludes with an infinite family of lattice rail arcs with calculated lattice stick numbers, Figure \ref{fig:cubicfamily}.

\begin{figure}[h]

\includegraphics[scale=.6]{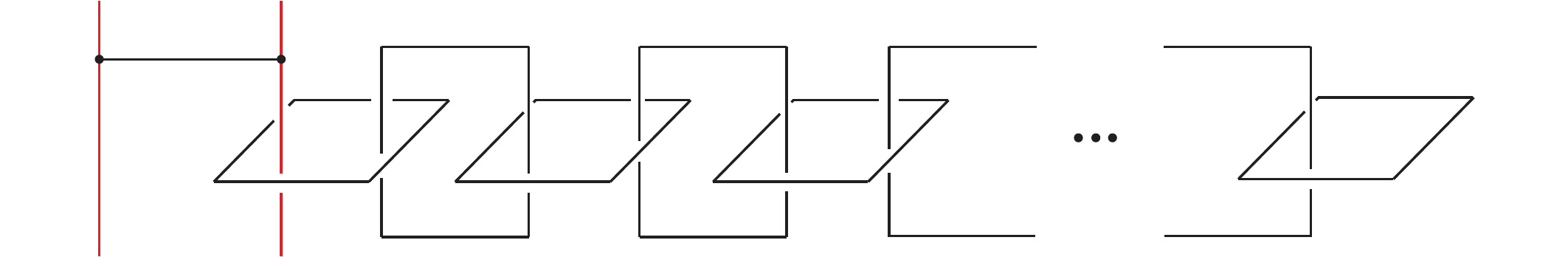}

\caption{Multi-component lattice rail arcs with $n$ unknot components and a  lattice stick number of $4n+1$.}
\label{fig:cubicfamily}
\end{figure}

\section*{Acknowledgements}

I would like to thank Clayton Shonkwiler for his correspondence and comments including Remark \ref{rmk:clayton}.

   \bibliographystyle{amsplain}
            \bibliography{railarc}

\end{document}